\documentclass[english]{smfart}
\usepackage[all]{xy}
\usepackage{color}
\usepackage{amsthm}
\usepackage{amssymb}
\usepackage[colorlinks=true]{hyperref}

\numberwithin{equation}{subsection}

\newtheorem{theorem}{Theorem}[section]
\newtheorem*{theorem*}{Theorem}
\newtheorem{lemma}[theorem]{Lemma}
\newtheorem{proposition}[theorem]{Proposition}
\newtheorem{corollary}[theorem]{Corollary}
\newtheorem*{corollary*}{Corollary}

\theoremstyle{remark}
\newtheorem{definition}[theorem]{Definition}
\theoremstyle{remark}
\newtheorem{example}[theorem]{Example}

\theoremstyle{remark}
\newtheorem{remark}[theorem]{Remark}
\theoremstyle{remark}
\newtheorem{notation}[theorem]{Notation}
\DeclareMathOperator{\Mot}{Mot}
\DeclareMathOperator{\Mix}{KPM}
\DeclareMathOperator{\Mixtr}{KMM}
\DeclareMathOperator{\hocolim}{hocolim}

\DeclareMathOperator{\incl}{incl}
\DeclareMathOperator{\id}{id}
\newcommand{\Ho}{\mathsf{Ho}}
\newcommand{\ko}{\: , \;}
\newcommand{\too}{\longrightarrow}
\newcommand{\gm}{\mathsf{gm}}
\newcommand{\abs}{\mathsf{abs}}
\newcommand{\dg}{\mathsf{dg}}
\newcommand{\dgHo}{\mathsf{H}^0}

\newcommand{\cAo}{{\mathcal A}^\op}

\newcommand{\cBo}{{\mathcal B}^\op}

\newcommand{\cA}{{\mathcal A}}
\newcommand{\cB}{{\mathcal B}}
\newcommand{\cC}{{\mathcal C}}
\newcommand{\cD}{{\mathcal D}}
\newcommand{\cE}{{\mathcal E}}

\newcommand{\cK}{{\mathcal K}}
\newcommand{\cL}{{\mathcal L}}
\newcommand{\cM}{{\mathcal M}}
\newcommand{\cN}{{\mathcal N}}

\newcommand{\cR}{{\mathcal R}}
\newcommand{\cS}{{\mathcal S}}
\newcommand{\cT}{{\mathcal T}}
\newcommand{\cU}{{\mathcal U}}
\newcommand{\cV}{{\mathcal V}}

\newcommand{\cW}{{\mathcal W}}

\newcommand{\bbC}{\mathbb{C}}
\newcommand{\bbD}{\mathbb{D}}

\newcommand{\bbL}{\mathbb{L}}

\newcommand{\bbK}{I\mspace{-6.mu}K}
\newcommand{\bbR}{\mathbb{R}}

\newcommand{\bbT}{\mathbb{T}}
\newcommand{\bbN}{\mathbb{N}}
\newcommand{\bbZ}{\mathbb{Z}}

\newcommand{\op}{\mathsf{op}} 
\newcommand{\ie}{\textsl{i.e.}\ }
\newcommand{\eg}{\textsl{e.g.}}

\newcommand{\Hmo}{\mathsf{Hmo}}

\newcommand{\tri}{\mathsf{tri}} 
\newcommand{\perf}{\mathsf{perf}} 

\newcommand{\Flat}{\mathsf{flat}}
\newcommand{\f}{\mathsf{f}}
\newcommand{\flt}{\mathsf{flt}}
\newcommand{\p}{\mathsf{p}}
\newcommand{\loc}{\mathsf{loc}}
\newcommand{\sSet}{\mathsf{sSet}}
\newcommand{\Set}{\mathsf{Set}}
\newcommand{\Map}{\mathsf{Map}}
\newcommand{\Cat}{\mathsf{Cat}} 
\newcommand{\CAT}{\mathsf{CAT}} 
\newcommand{\Hom}{\mathsf{Hom}} 
\newcommand{\rep}{\mathsf{rep}} 
\newcommand{\Loc}{\mathsf{L}}
\newcommand{\St}{\mathsf{St}}
\newcommand{\stab}{\mathsf{stab}}
\newcommand{\Fun}{\mathsf{Fun}} 
\newcommand{\pref}[1]{\smash{\widehat{#1}}}
\newcommand{\spref}[1]{\smash{\mathsf{s}\, \widehat{#1}}}
\newcommand{\dgcat}{\mathsf{dgcat}}
\newcommand{\HO}{\mathsf{HO}} 

\newcommand{\Mloc}{\Mot^{\mathsf{loc}}_{\dg}}
\newcommand{\Mlocu}{\Mot^{\mathsf{uloc}}_{\dg}}
\newcommand{\cMloc}{\cM ot^{\mathsf{loc}}_{\dg}}
\newcommand{\cMlocu}{\cM ot^{\mathsf{uloc}}_{\dg}}
\newcommand{\Uloc}{\cU^{\mathsf{loc}}_{\dg}}

\newcommand{\Spt}{\mathsf{Sp}}
\newcommand{\Sp}{\mathsf{Sp}^{\Sigma}}

\newcommand{\uHom}{\underline{\mathsf{Hom}}}
\newcommand{\HomC}{\uHom_{!}}
\newcommand{\HomL}{\uHom_{\mathsf{loc}}}

\newcommand{\dgS}{\cS}
\newcommand{\dgD}{\cD}

\newcommand{\internalcomment}[1]{}

\title[Symmetric monoidal structure on non-commutative motives]
{Symmetric monoidal structure on \\non-commutative motives}
\alttitle{Monoidal structure on non-commutative motives}
\author[D.-C.~Cisinski]{Denis-Charles Cisinski}
\thanks{The first named author was partially supported by the ANR (grant No.~ANR-07-BLAN-042). The second named author was partially supported by the Est{\'i}mulo {\`a} Investiga{\c c}{\~a}o Award 2008 - Calouste Gulbenkian Foundation.}
\address{Institut Galil{\'e}e\\ Universit{\'e} Paris 13\\
99, Av. Jean-Baptiste Cl{\'e}ment\\ 93430 Villetaneuse\\ France}
\email{cisinski@math.univ-paris13.fr}
\urladdr{http://www-math.univ-paris13.fr/~cisinski/}
\author[G.~Tabuada]{Gon{\c c}alo~Tabuada}
\address{Departamento de Matem{\'a}tica e CMA, FCT-UNL\\ Quinta da Torre\\ 2829-516 Caparica\\ Portugal}
\email{tabuada@fct.unl.pt}
\date{\today}


\begin{document}
\frontmatter
\begin{abstract}
In this article we further the study of non-commutative motives, initiated in \cite{CT,Duke}.
Our main result is the construction of a symmetric monoidal structure on the localizing motivator
\smash{$\Mloc$} of dg categories. As an application, we obtain\,: (1) a computation of the spectra of
morphisms in \smash{$\Mloc$} in terms of non-connective algebraic  $K$-theory; (2) a fully-faithful embedding
of Kontsevich's category $\Mixtr_k$ of non-commutative mixed motives into the
base category \smash{$\Mloc(e)$} of the
localizing motivator;
{(3) a simple construction of the Chern character maps from non-connective algebraic $K$-theory to
negative and periodic cyclic homology}; (4) a precise connection between To{\"e}n's secondary
$K$-theory and the Grothendieck ring of $\Mixtr_k$; (5) a description of the
Euler characteristic {in $\Mixtr_k$} in terms of Hochschild homology.
\end{abstract}
\maketitle
\tableofcontents

\section*{Introduction}
\subsection*{Dg categories}
A {\em differential graded (=dg) category}, over a commutative base ring $k$, is a category enriched over 
complexes of $k$-modules (morphisms sets are such complexes)
in such a way that composition fulfills the Leibniz rule\,:
$d(f\circ~g)=(df)\circ g+(-1)^{\textrm{deg}(f)}f\circ(dg)$.
Dg categories enhance and solve many of the technical problems inherent to triangulated categories;
see Keller's ICM adress~\cite{ICM}. In {\em non-commutative algebraic geometry} in the sense of
Bondal, Drinfeld, Kapranov, Kontsevich, To{\"e}n,
Van den Bergh, $\ldots$ \cite{BKap,Bvan,Drinfeld,Chitalk,IAS,Kontsevich-Langlands,Lattice,Toen},
they are considered as dg-enhancements of derived categories of (quasi-)coherent sheaves on a
hypothetic non-commutative space.

\subsection*{Localizing invariants}
All the classical (functorial) invariants, such as Hochschild homology, cyclic homology and its variants
(periodic, negative, $\ldots$), algebraic $K$-theory, and even topological Hochschild homology and
topological cyclic homology (see \cite{AGT}), extend naturally from $k$-algebras to dg categories.
In order to study {\em all} these classical invariants simultaneously, the second named author
introduced in \cite{Duke} the notion of {\em localizing invariant}. This notion, that we now recall,
makes use of the language of Grothendieck derivators~\cite{Grothendieck}, a formalism which
allows us to state and prove precise universal properties; consult Appendix~\ref{appendix:A}.
Let $\mathit{L}: \HO(\dgcat) \to \bbD$ be a morphism of derivators, from the
derivator associated to the Morita model structure of dg categories (see \S\ref{sub:Modelstructure}), to
a triangulated derivator
(in practice, $\bbD$ will be the derivator associated to a stable model category $\cM$,
and $\mathit{L}$ will come from a functor $\dgcat\to\cM$ which sends derived Morita equivalences to weak
equivalences in $\cM$).
We say that $\mathit{L}$ is a {\em localizing invariant} (see Definition~\ref{not:Linv}) if it preserves filtered homotopy
colimits as well as the terminal object, and sends Drinfeld exact sequences of dg categories
\begin{eqnarray*}
\cA \too \cB \too \cC & \mapsto & \mathit{L}(\cA) \too \mathit{L}(\cB) \too \mathit{L}(\cC) \too \mathit{L}(\cA)[1]
\end{eqnarray*}
to distinguished triangles in the base category $\bbD(e)$ of $\bbD$.
Thanks to the work of Keller~\cite{Exact,Exact1}, Schlichting~\cite{Marco},
Thomason-Trobaugh~\cite{Thomason}, and Blumberg-Mandell~\cite{Blum-Mandell},
all the mentioned invariants satisfy localization\footnote{In the case of algebraic
$K$-theory we consider its non-connective version.}, and so give rise to localizing invariants. 
In \cite{Duke}, the second named author proved that there exists a localizing invariant
$$ \Uloc: \HO(\dgcat) \too \Mloc\,,$$
with values in a strong triangulated derivator (see \S\ref{sub:properties}),
such that given any strong triangulated derivator $\bbD$, we have an induced equivalence of categories
\begin{equation*}
\label{eq:intro} (\Uloc)^{\ast}: \HomC(\Mloc, \bbD) \stackrel{\sim}{\too} \HomL(\HO(\dgcat), \bbD)\,.
\end{equation*}
The left-hand side denotes the category of homotopy colimit preserving morphisms of derivators,
and the right-hand side the category of localizing invariants.

Because of this universality property, which is a reminiscence of motives,
$\Uloc$ is called the {\em universal localizing invariant}, $\Mloc$ the {\em localizing motivator},
and the base category $\Mloc(e)$ of the localizing motivator the
{\em category of non-commutative motives over $k$}.

\subsection*{Symmetric monoidal structure}
The purpose of this article is to develop a new ingredient in the theory of non-commutative motives:
{\em symmetric monoidal structures}. The tensor product extends naturally from $k$-algebras to
dg categories, giving rise to a symmetric monoidal structure on $\HO(\dgcat)$; see Theorem~\ref{thm:Mdgcat}.
Therefore, it is natural to consider localizing invariants which are symmetric monoidal. {Examples include Hochschild homology and the mixed and periodic complex constructions; see Examples~\ref{ex:HH}-\ref{ex:Periodiccomp}}. The main result of this article is the following.
\begin{theorem}{(see Theorem~\ref{thm:main-mon})}\label{thm:intro1}
The localizing motivator $\Mloc$ carries a canonical symmetric monoidal structure $-\otimes^{\bbL}-$,
making the universal localizing
invariant $\Uloc$ symmetric monoidal. Moreover, this tensor product preserves homotopy colimits
in each variable and is characterized by the following universal property: given any strong
triangulated derivator $\bbD$, endowed with a monoidal structure which preserves homotopy colimits,
we have an induced equivalence
of categories
\begin{equation*}
(\Uloc)^{\ast}: \HomC^{\otimes}(\Mloc, \bbD) \stackrel{\sim}{\too} \HomL^{\otimes}(\HO(\dgcat), \bbD)\,,
\end{equation*}
where the left-hand side stands for the category of symmetric monoidal homotopy colimit
preserving morphisms of derivators, while the right-hand side stands for the
category symmetric monoidal morphisms of derivators which are also localizing
invariants; see \S\ref{sub:M1}. Furthermore, $\Mloc$ admits an explicit symmetric monoidal Quillen model.
\end{theorem}
The proof of Theorem~\ref{thm:intro1} is based on an alternative
description of $\Mloc$, with two complementary aspects:
a {\em constructive} one, and another given by \emph{universal properties}.

The constructive aspect, \ie the construction of an explicit symmetric monoidal
Quillen model for $\Mloc$, is described in the main body of the text. The key starting
point is the fact that homotopically finitely presented dg categories
are stable under derived tensor product; see Theorem~\ref{thm:tensorstable}. This allows us to obtain
a small symmetric monoidal category which ``generates'' the entire Morita
homotopy category of dg categories. Starting from this
small monoidal category, we then construct a specific symmetric monoidal Quillen model for each one
of the derivators used in the construction of $\Mloc$; see \S\ref{sub:proof-main}. 

The characterization of $\Mloc$ by its universal property (as stated in Theorem~\ref{thm:intro1}),
relies on general results and constructions in the theory of Grothendieck derivators, and is described in the appendix. We develop some general results concerning the behavior of
monoidal structures under classical operations\,: Kan extension (see Theorem~\ref{monoidalderKanext}),
left Bousfield localization (see Proposition~\ref{derivatorleftlocmonoidal}) and stabilization
(see Theorem~\ref{stableDay}). Using these general results, we then characterize by a precise
universal property each one of the Quillen models used in the construction of the new symmetric
monoidal Quillen model for $\Mloc$; see \S\ref{sub:proof-main}.

Let us now describe some applications of Theorem~\ref{thm:intro1}.

\subsection*{Non-connective $K$-theory}
As mentioned above, non-connective algebraic $K$-theory $\bbK(-)$ is an example of a localizing
invariant. In \cite{CT} the authors proved that this invariant becomes co-representable in $\Mloc(e)$
by the unit object $\Uloc(\underline{k})$ (where $\underline{k}$ corresponds to $k$, seen
as a dg~category with one object). In other words, given any dg category $\cA$, we have a natural
isomorphism in the stable homotopy category of spectra\,:
\begin{equation}\label{eq:intro2}
\bbR\Hom(\,\Uloc(\underline{k}),\, \Uloc(\cA)\,) \simeq \bbK(\cA)\,.
\end{equation}
A fundamental problem of the theory of non-commutative motives is the computation of the
(spectra of) morphisms in the category of non-commutative motives between any two objects.
Using the monoidal structure of Theorem~\ref{thm:intro1} we extend the above natural
isomorphism (\ref{eq:intro2}), and thus obtain a partial solution to this problem.
 \begin{theorem}{(see Theorem~\ref{thm:co-repres-ext})}\label{thm:intro2}
Let $\cB$ be a saturated dg category in the sense of Kontsevich, \ie its complexes of
morphisms are perfect and $\cB$ is perfect as a bimodule over itself; see
Definition~\ref{def:saturated}. Then, for every small dg category $\cA$, we have
a natural isomorphism in the stable homotopy category of spectra
\begin{equation*}
\bbR\Hom(\,\Uloc(\cB),\, \Uloc(\cA)\,) \simeq \bbK(\rep(\cB, \cA))\,,
\end{equation*}
where $\rep(-,-)$ denotes the internal Hom-functor in the Morita homotopy category of dg categories;
see \S\ref{sub:Monoidstructure}.
\end{theorem}
Given a quasi-compact and separated $k$-scheme $X$, there is a natural dg category $\perf(X)$
which enhances the category of perfect complexes (\emph{i.e.} of compact objects)
in the (unbounded) derived category $\cD_{qcoh}(X)$
of quasi-coherent sheaves on $X$; see Example~\ref{ex:Toen}. Moreover, when $X$ is smooth and
proper, the dg category $\perf(X)$ is a saturated dg category. In this geometrical situation, we
have the following computation.
\begin{proposition}{(see Proposition~\ref{prop:co-repres-schemes})}
Given smooth and proper $k$-schemes $X$ and $Y$, we have a natural isomorphism in the stable homotopy
category of spectra
$$\bbR\Hom(\, \Uloc(\perf(X)), \Uloc(\perf(Y)) \,) \simeq \bbK(X \times Y)\,,$$
where $\bbK(X \times Y)$ denotes the non-connective algebraic $K$-theory spectrum of $X \times Y$.
\end{proposition}
\subsection*{Kontsevich's category of non-commutative mixed motives}
In his {\em non-commutative algebraic geometry} program \cite{IAS,Kontsevich-Langlands,Lattice},
Kontsevich introduced the category $\Mixtr_k$ of {\em non-commutative mixed motives};
see \S\ref{sub:Kontsevich}. Roughly, $\Mixtr_k$ is obtained by taking a formal Karoubian
triangulated envelope of the category of saturated dg categories (with algebraic $K$-theory of
bimodules as morphism sets). Using Theorem~\ref{thm:intro2}, we prove the following result.
\begin{proposition}{(see Proposition~\ref{prop:Kontsevich})}\label{prop:intro2}
There is a natural fully-faithful embedding (enriched over spectra) of Kontsevich's category
$\Mixtr_k$ of non-commutative mixed motives into the base category $\Mloc(e)$ of the
localizing motivator. The essential image is the thick triangulated subcategory spanned by motives of saturated dg categories.
\end{proposition}
Note that, in contrast with Kontsevich's {\em ad hoc} definition, the category $\Mloc(e)$
of non-commutative motives is defined purely in terms of precise universal properties.

\subsection*{Chern characters}
Let $$ \mathit{E}: \HO(\dgcat) \too \bbD$$
be a symmetric monoidal localizing invariant. Thanks to Theorem~\ref{thm:intro1} there is a (unique) symmetric monoidal homotopy colimit preserving morphism of derivators $\mathit{E}_{\gm}$ which 
makes the diagram
$$
\xymatrix{
\HO(\dgcat) \ar[r]^-{\mathit{E}} \ar[d]_-{\Uloc} & \bbD \\
\Mloc \ar[ur]_{\mathit{E}_{\gm}} & 
}
$$
commute (up to unique $2$-isomorphism).
We call $\mathit{E}_{\gm}$ the {\em geometric realization}
of $\mathit{E}$. If $\mathit{E}(\underline{k})\simeq {\bf 1}$
denotes the unit of $\bbD$, we can also associate to $\mathit{E}$ its {\em absolute realization}
\begin{eqnarray*}\mathit{E}_{\abs}:=\bbR\Hom_{\bbD}({\bf 1}, \mathit{E}_{\gm}(-))&&
\text{(see Definition \ref{not:geom-real})}
\end{eqnarray*}
\begin{proposition}{(see Proposition~\ref{prop:Chern})}\label{prop:Chern-intro}
The geometric realization of $\mathit{E}$ induces a canonical Chern character
\begin{equation*}
\bbK(-) \Rightarrow \bbR\Hom({\bf 1},E(-))\simeq \mathit{E}_{\abs}(\Uloc(-))\,.
\end{equation*} 
Here, $\bbK(-)$ and $\mathit{E}_{\abs}(\Uloc(-))$ are two morphisms of derivators defined on $\HO(\dgcat)$.
\end{proposition}
Let $\cA$ be a small dg category. When $\mathit{E}$ is given by the
mixed complex construction (see Example~\ref{ex:Mixedcomp}), the absolute realization of $\Uloc(\cA)$
identifies with the negative cyclic homology complex $\mathit{HC}^-(\cA)$ of $\cA$.
Therefore by Proposition~\ref{prop:Chern-intro}, we obtain a canonical Chern character
$$ \bbK(-) \Rightarrow \mathit{HC}^-(-)$$
from non-connective $K$-theory to negative cyclic homology; see Example~\ref{ex:negative}.
When $\mathit{E}$ is given by the composition of the mixed complex construction with the
periodization procedure (see Example~\ref{ex:Periodiccomp}), the absolute realization of
$\Uloc(\cA)$ identifies with the periodic cyclic homology complex $\mathit{HP}(\cA)$ of $\cA$.
Therefore by Proposition~\ref{prop:Chern-intro}, we obtain a canonical Chern character
$$ \bbK(-) \Rightarrow \mathit{HP}(-)$$
from non-connective $K$-theory to periodic cyclic homology; see Example~\ref{ex:periodic}.

\subsection*{To{\"e}n's secondary $K$-theory}
In his {\em sheaf categorification} program \cite{NW-Toen, Abel}, To{\"e}n introduced a
``categorified'' version of algebraic $K$-theory, named {\em secondary $K$-theory}.
Given a commutative ring $k$, the {\em secondary $K$-theory ring $K_0^{(2)}(k)$ of $k$}
is roughly the quotient of the free abelian group on derived Morita isomorphism classes of saturated
dg categories, by the relations $[\cB]=[\cA]+[\cC]$ coming from Drinfeld exact sequences $\cA \to \cB \to \cC$.
The multiplication is induced from the derived tensor product; see \S\ref{sub:secondary}.
As To{\"e}n pointed out in \cite{NW-Toen}, one of the motivations for the study of this secondary
$K$-theory is its expected connection with an hypothetical Grothendieck ring of motives in the
non-commutative setting. Thanks to the monoidal structure of Theorem~\ref{thm:intro1} we are
now able to make this connection precise.
\begin{definition}{(see Definition~\ref{def:motivic-ring})}\label{def:intro}
Given a commutative ring $k$, the {\em Grothendieck ring $\cK_0(k)$ of non-commutative motives
over $k$} is the Grothendieck ring of Kontsevich mixed motives (\ie of the thick triangulated subcategory of $\Mloc(e)$ generated by the objects $\Uloc(\cA)$, where $\cA$
runs over the family of saturated dg categories).
\end{definition}   
Kontsevich's saturated dg categories can be
characterized conceptually as the dualizable objects in the Morita homotopy category;
see Theorem~\ref{thm:Toen-dualizable}. Therefore, since the universal localizing
invariant $\Uloc$ is symmetric monoidal we obtain a ring homomorphism
\begin{equation*}
\Phi(k): K_0^{(2)}(k) \too \cK_0(k)\,.
\end{equation*}
Moreover, the Grothendieck ring of Definition~\ref{def:intro}
is non-trivial (see Remark~\ref{rk:non-trivial}), functorial in $k$ (see Remark~\ref{rk:functoriality}),
and {the ring homomorphism $\Phi(k)$ is functorial in $k$ and surjective ``up to cofinality''} (see Remark~\ref{rk:connection}). Furthermore, any {\em realization} of $K_0^{(2)}(k)$ (\ie ring homomorphism $K_0^{(2)}(k) \to R$),
which is induced from a symmetric monoidal localizing invariant, factors through $\Phi(k)$.
An interesting example is provided by To{\"e}n's rank map
\begin{eqnarray*}rk_0: K_0^{(2)}(k) \too K_0(k)&&
\text{(see Remark \ref{rk:connection}).}
\end{eqnarray*}
\subsection*{Euler characteristic}
Recall that, in any symmetric mo\-noi\-dal category, we have
the notion of {\em Euler characteristic $\chi(X)$} of a dualizable object $X$
(see Definition~\ref{def:Euler}).
In the symmetric monoidal category of non-commutative motives we have the following computation.

\begin{proposition}{(see Proposition~\ref{prop:Euler})}\label{prop:intro3}
Let $\cA$ be a saturated dg category. Then, $\chi(\Uloc(\cA))$ is the element of the
Grothendieck group $K_0(k)$ which is associated to the (perfect)
Hochschild homology complex $\mathit{HH}(\cA)$ of $\cA$.
\end{proposition}

When $k$ is the field of complex numbers, the Grothendieck ring $K_0(\bbC)$ is naturally
isomorphic to $\bbZ$ and the Hochschild homology of a smooth and proper $k$-scheme
$X$ agrees with the Hodge cohomology
$H^{\ast}(X, \Omega^{\ast}_X)$ of $X$. Therefore, when we work over $\bbC$,
and if $\perf(X)$ denotes the (saturated) dg category of perfect complexes over $X$, the Euler
characteristic of $\Uloc(\perf(X))$ is the classical Euler characteristic of $X$.

\medbreak\noindent\textbf{Acknowledgments\,:} The authors are very grateful to Bertrand To{\"e}n
for sharing his insights on saturated dg categories {and to Maxim Kontsevich for comments on a previous version which motivated the creation of subsection~\ref{sub:Chern}.} The second named author
would like also to thank Paul Balmer and Christian Haesemeyer for some conversations.
This article was initiated at the Institut Henri Poincar{\'e} in Paris. The authors would like to
thank this institution for its hospitality
and stimulating working environment.
\mainmatter
\section{Preliminaries}\label{sec:preliminaries}

\subsection{Notations}\label{sub:notation}
Throughout the article we will work over a fixed commutative and unital base ring $k$. 

We will denote by $\cC(k)$ be the category of (unbounded) complexes of $k$-modules;
see~\cite[\S2.3]{Hovey}. We will use co-homological notation, \ie the differential increases the degree.
The category $\cC(k)$ is a symmetric monoidal model category (see \cite[Definition\,4.2.6]{Hovey}),
where one uses the {\em projective} model structure for which weak equivalences are quasi-isomorphisms
and fibrations are surjections; see~\cite[Proposition\,4.2.13]{Hovey}.

The category of sets will be denoted by $\mathsf{Set}$, the category of
simplicial sets by $\sSet$, and the category of pointed simplicial sets by $\sSet_{\bullet}$;
see~\cite[\S I]{Jardine}. The categories $\sSet$ and $\sSet_{\bullet}$ are symmetric monoidal model
categories; see~\cite[Proposition\,4.2.8]{Jardine}. The weak equivalences are the maps whose
geometric realization is a weak equivalence of topological spaces, the fibrations are the Kan-fibrations,
and the cofibrations are the inclusion maps.

We will denote by $\Spt^{\bbN}$ the category of spectra and by $\Sp$ the category of symmetric spectra
(of pointed simplicial sets); see~\cite{HSS}.

Finally, the adjunctions will be displayed vertically with the left (resp. right) adjoint on the left- (resp. right-) hand side. 

\subsection{Triangulated categories}\label{sub:triangulated}
Throughout the article we will use the language of triangulated categories.
The reader unfamiliar with this language is invited to consult Neeman's book~\cite{Neeman}
or Verdier's original monograph~\cite{Verdier}. Recall from \cite[Definition~4.2.7]{Neeman}
that given a triangulated category $\cT$ admitting arbitrary small coproducts, an object $G$ in $\cT$
is called {\em compact} if for each family $\{Y_i\}_{i \in I}$ of objects in $\cT$, the canonical morphism
$$ \underset{i \in I}{\bigoplus} \,\Hom_{\cT}(G,Y_i) \stackrel{\sim}{\too}
\Hom_{\cT}(G, \underset{i \in I}{\bigoplus} \,Y_i)$$
is invertible. We will denote by $\cT_c$ the category of compact objects in $\cT$. 

\subsection{Quillen model categories}\label{sub:Qmodel}
Throughout the article we will use freely the language of Quillen model categories.
The reader unfamiliar with this language is invited to consult Goerss~\&~Jardine~\cite{Jardine},
Hirschhorn~\cite{Hirschhorn}, Hovey~\cite{Hovey}, or Quillen's original monograph~\cite{Quillen}.
Given a model category $\cM$, we will denote by $\Ho(\cM)$ its homotopy category and~by 
$$\Map(-,-): \Ho(\cM)^\op \times \Ho(\cM) \too \Ho(\sSet)$$ 
its homotopy function complex; see~\cite[Definition~17.4.1]{Hirschhorn}. 
\subsection{Grothendieck derivators}
Throughout the article we will use the language of Grothendieck derivators. Derivators allow us to
state and prove precise universal properties and to dispense with many of the technical
problems one faces in using Quillen model categories. Since this language may be less
familiar to the reader, we revise it in the appendix. 
Given a model category $\cM$, we will denote by $\HO(\cM)$ its associated derivator;
see \S\ref{sub:derivators}. Any triangulated derivator $\bbD$ is canonically
enriched over spectra; see \S\ref{sub:spectral-enr}. We will denote by $\bbR\Hom_{\bbD}(X,Y)$
the spectrum of maps from $X$ to $Y$ in $\bbD$. When there is no ambiguity, we
will write $\bbR\Hom(X,Y)$ instead of $\bbR\Hom_{\bbD}(X,Y)$.

\section{Background on dg categories}
In this section we collect and recall the notions and results on the (homotopy) theory of dg categories
which will be used throughout the article; see \cite{IMRN, cras, Toen}. At the end of the
section we prove a new result: the derivator associated to the Morita model structure carries
a symmetric monoidal structure; see Theorem~\ref{thm:Mdgcat}. This result is the starting point
of the article, since it allow us to define symmetric monoidal localizing invariants;
see Definition~\ref{not:Linv-mon}.

\begin{definition}\label{def:dgcategory}
A {\em small dg category $\cA$} is a $\cC(k)$-category;
see~\cite[Definition\,6.2.1]{Borceaux}. Recall that this consists of the following data\,: 
\begin{itemize}
\item[-] a set of objects $\mathrm{obj}(\cA)$ (usually denoted by $\cA$ itself); 
\item[-] for each ordered pair of objects $(x,y)$ in $\cA$, a complex of $k$-modules $\cA(x,y)$; 
\item[-] for each ordered triple of objects $(x,y,z)$ in $\cA$, a composition morphism in $\cC(k)$
$$\cA(y,z)\otimes \cA(x,y) \too \cA(x,z)\,,$$
satisfying the usual associativity condition; 
\item[-] for each object $x$ in $\cA$, a morphism $k \to \cA(x,x)$, satisfying the usual unit condition
with respect to the above composition.
\end{itemize}
\end{definition}
\begin{definition}\label{def:dgfunctor}
A {\em dg functor} $F: \cA \to \cB$ is a $\cC(k)$-functor; see~\cite[Definition\,6.2.3]{Borceaux}.
Recall that this consists of the following data\,:
\begin{itemize}
\item[-] a map of sets $F: \mathrm{obj}(\cA) \too \mathrm{obj}(\cB)$;
\item[-] for each ordered pair of objects $(x,y)$ in $\cA$, a morphism in $\cC(k)$
$$ F(x,y): \cA(x,y) \too \cB(Fx, Fy)\,,$$
satisfying the usual unit and associativity conditions.
\end{itemize}
\end{definition}
\begin{notation}\label{not:dgcat}
We denote by $\dgcat$ the category of small dg categories.
\end{notation}
\subsection{Dg cells}\label{sub:dg-cells}
\begin{itemize}
\item[(i)] Let $\underline{k}$ be the dg category with one
object $\ast$ and such that $\underline{k}(\ast,\ast):=k$ (in degree
zero). Note that given a small dg category $\cB$, there is a bijection between the set of
dg functors from $\underline{k}$ to $\cB$ and the set of objects of $\cB$.
\item[(ii)] For $n \in \mathbb{Z}$, let $S^{n}$ be the complex $k[n]$
(with $k$ concentrated in degree $n$) and $D^n$ the mapping
cone on the identity of $S^{n-1}$. We denote by $\dgS(n)$ the dg
category with two objects $1$ and $2$ such that $ \dgS(n)(1,1)=k \ko
\dgS(n)(2,2)=k \ko \dgS(n)(2,1)=0  \ko \dgS(n)(1,2)=S^{n} $ and
composition given by multiplication. We denote by $\dgD(n)$ the dg
category with two objects $3$ and $4$ such that $ \dgD(n)(3,3)=k \ko
\dgD(n)(4,4)=k \ko \dgD(n)(4,3)=0 \ko \dgD(n)(3,4)=D^n $ and
composition given by multiplication. 
\item[(iii)] For $n \in \bbZ$, let $\iota(n):\dgS(n-1)\to \dgD(n)$ be the dg functor that sends $1$ to
$3$, $2$ to $4$ and $S^{n-1}$ to $D^n$ by the identity on $k$ in
degree~$n-1$\,:
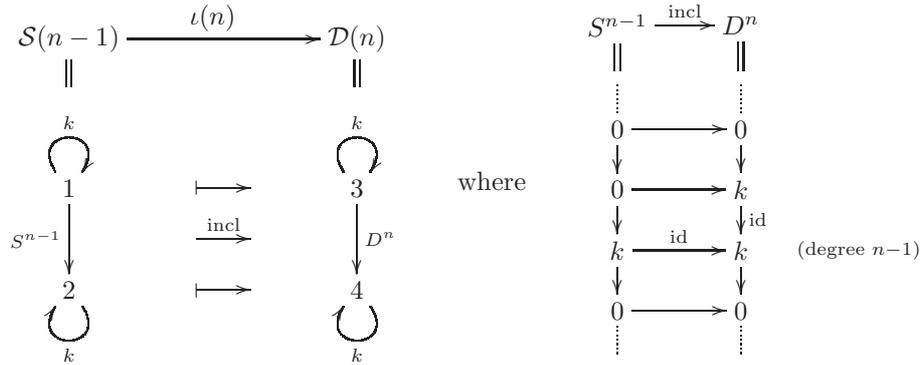
$$
\vcenter{
\xymatrix@C=2em@R=1em{
\dgS(n-1) \ar@{=}[d] \ar[rrr]^{\displaystyle \iota(n)}
&&& \dgD(n) \ar@{=}[d]
\\
&&&\\
\\
1 \ar@(ul,ur)[]^{k} \ar[dd]_-{S^{n-1}}
& \ar@{|->}[r] &
& 3\ar@(ul,ur)[]^{k}  \ar[dd]^-{D^n}
\\
& \ar[r]^-{\incl} &
\\
2 \ar@(dr,dl)[]^{k}
& \ar@{|->}[r] &
& 4\ar@(dr,dl)[]^{k}
}}
\qquad\text{where}\qquad
\vcenter{\xymatrix@R=1em@C=.8em{ S^{n-1} \ar[rr]^-{\incl} \ar@{=}[d]
&& D^n \ar@{=}[d]
\\
\ar@{.}[d]
&& \ar@{.}[d]
\\
0 \ar[rr] \ar[d]
&& 0 \ar[d]
\\
0 \ar[rr] \ar[d]
&& k \ar[d]^{\id}
\\
k \ar[rr]^{\id} \ar[d]
&& k \ar[d]
&{\scriptstyle(\textrm{degree }n-1)}
\\
0 \ar[rr] \ar@{.}[d]
&& 0 \ar@{.}[d]
\\
&&}}
$$
\end{itemize}
\begin{notation}\label{not:I-cell}
Let $I$ be the set consisting of the dg functors
$\{\iota(n)\}_{n\in \mathbb{Z}}$ and the dg functor $\emptyset
\rightarrow \underline{k}$ (where the empty dg category $\emptyset$
is the initial object in $\dgcat$).
\end{notation}
\begin{definition}\label{def:dg-cell}
A dg category $\cA$ is called a \emph{dg cell} (resp. a  \emph{finite dg cell})
if the unique dg functor $\emptyset\to\cA$ can be expressed as a transfinite
(resp. finite) composition of pushouts of
dg functors in $I$; see~\cite[Definition~10.5.8(2)]{Hirschhorn}.
\end{definition}
\subsection{Dg modules}\label{sub:modules}
Let $\cA$ be a small (fixed) dg category. 
\begin{definition}\label{def:modules}
\begin{itemize}
\item[-] The category $\dgHo(\cA)$ has the same objects as $\cA$ and
morphisms given by $\dgHo(\cA)(x,y):=\textrm{H}^0(\cA(x,y))$, where $\textrm{H}^0(-)$
denotes the $0$-th co-homology group functor. 
\item[-] The {\em opposite} dg category $\mathcal{A}^{\op}$ of $\cA$ has the same objects
as $\mathcal{A}$ and complexes of morphisms given by
$\mathcal{A}^{\op}(x,y):=\mathcal{A}(y,x)$. 
\item[-] A {\em right dg $\cA$-module} $M$ (or simply a $\cA$-module) is a dg functor
$M:\cA^{\op} \to \cC_{\dg}(k)$ with values in the dg category $\cC_{\dg}(k)$ of complexes of
$k$-modules.
\end{itemize}
\end{definition}
We denote by $\cC(\cA)$ the category of right dg $\cA$-modules. Its objects are the right dg
$\cA$-modules and its morphisms are the natural transformations of dg functors;
see \cite[\S 2.3]{ICM}. The differential graded structure of $\cC_{\dg}(k)$ makes $\cC(\cA)$
naturally into a dg category; see \cite[\S 3.1]{ICM} for details. We denote by $\cC_{\dg}(\cA)$
this dg category of right dg $\cA$-modules.
Recall from~\cite[Theorem~3.2]{ICM} that $\cC(\cA)$ carries a standard {\em projective}
model structure. A morphism $M \to M'$ in $\cC(\cA)$ is a {\em weak equivalence},
resp. a {\em fibration}, if for any object $x$ in $\cA$, the induced morphism $M(x) \to M'(x)$
is a weak equivalence, resp. a fibration, in the projective model structure on $\cC(k)$.
In particular, every object is fibrant. Moreover, the dg category $\cC_{\dg}(\cA)$ endowed
with this model structure is a {\em $\cC(k)$-model category} in the sense of~\cite[Definition~4.2.18]{Hovey}.
\begin{notation}
\begin{itemize}
\item[-] We denote by $\cD_{\dg}(\cA)$ the full dg subcategory of fibrant
and cofibrant objects in $\cC_{\dg}(\cA)$.
\item[-] We denote by $\cD(\cA)$ the {\em derived category of $\cA$}, \ie the localization of
$\cC(\cA)$ with respect to the class of weak equivalences. Note that $\cD(\cA)$ is a triangulated
category with arbitrary small coproducts.
Recall from \S\ref{sub:triangulated} that we denote by $\cD_c(\cA)$
the category of compact objects in $\cD(\cA)$.
\end{itemize}
\end{notation}
Since $\cC_{\dg}(\cA)$ is a $\cC(k)$-model category, \cite[Proposition~3.5]{Toen}
implies that we have a natural equivalence of (triangulated) categories
\begin{equation}\label{eq:equi-cat}
\dgHo(\cD_{\dg}(\cA))\simeq\cD(\cA)\,.
\end{equation}
We denote by
$$
\begin{array}{lcr}
\underline{h}: \cA  \too   \cC_{\dg}(\cA) && x  \mapsto \cA(-,x)\,,
\end{array}
$$
the classical {\em Yoneda dg functor}; see \cite[\S6.3]{Borceaux}. Since the $\cA$-modules
$\cA(-,x)$, $x$ in $\cA$, are fibrant and cofibrant in the projective model structure, the Yoneda
functor factors through the inclusion $\cD_{\dg}(\cA)\subset \cC_{\dg}(\cA)$.
Thanks to the above equivalence (\ref{eq:equi-cat}), we obtain an induced fully-faithful functor
\begin{eqnarray*}
\underline{h}:\dgHo(\cA) \hookrightarrow \cD(\cA) && x \mapsto \cA(-,x)\,.
\end{eqnarray*}
Finally, let $F: \cA \to \cB$ be a dg functor. As shown in \cite[\S 3]{Toen} it induces a
{\em restriction/extension of scalars} Quillen adjunction (on the left-hand side)
$$
\xymatrix{
\cC(\cB) \ar@<1ex>[d]^{F^{\ast}} && \cD(\cB) \ar@<1ex>[d]^{F^{\ast}}\\
\cC(\cA) \ar@<1ex>[u]^{F_!}  && \cD(\cA) \ar@<1ex>[u]^{\bbL F_!} \,,\\
}
$$
which can be naturally derived (on the right-hand side).

\subsection{Morita model structure}\label{sub:Modelstructure}

\begin{definition}\label{def:Morita}
A dg functor $F: \cA \to \cB$ is a called a {\em derived Morita equivalence}
if the restriction of scalars functor $F^{\ast}: \cD(\cB) \stackrel{\sim}{\to} \cD(\cA)$
is an equivalence of triangulated categories. Equivalently, $F$ is a derived Morita equivalence
if the extension of scalars functor $\bbL F_!: \cD(\cA) \stackrel{\sim}{\to} \cD(\cB)$
is an equivalence of triangulated categories.
\end{definition}

\begin{theorem}\label{thm:Morita}
The category $\dgcat$ carries a cofibrantly generated Quillen model structure
(called the {\em Morita model structure}), whose weak equivalences are the
derived Morita equivalences. Moreover, its set of generating cofibrations is the
set $I$ of Notation~\ref{not:I-cell}.
\end{theorem}

\begin{proof}
See \cite[Th{\'e}or{\`e}me~5.3]{IMRN}.
\end{proof}

\begin{notation}\label{not:hom-Morita}
We denote by $\Hmo$ the homotopy category hence obtained.
\end{notation}
Recall from \cite[Remarque~5.4]{IMRN} that the fibrant objects of the Morita model
structure (called the {\em Morita fibrant dg categories}) are the dg categories $\cA$ for
which the image of the induced functor
\begin{eqnarray*}
\underline{h}:\dgHo(\cA) \hookrightarrow \cD(\cA) && x \mapsto \cA(-,x)
\end{eqnarray*}
is stable under (co-)suspensions, cones and direct factors. 
\subsection{Monoidal structure}\label{sub:Monoidstructure}
\begin{definition}
Let $\cA_1$ and $\cA_2$ be small dg categories. The {\em tensor product}
$\cA_1 \otimes \cA_2$ of $\cA_1$ with $\cA_2$ is defined as follows:
the set of objects of $\cA_1 \otimes \cA_2$ is $\mbox{obj}(\cA_1)\times \mbox{obj}(\cA_2)$
and given objects $(x,x')$ and $(y,y')$ in $\cA_1 \otimes \cA_2$, we set
$$ (\cA_1 \otimes \cA_2)((x,x'),(y,y'))= \cA_1(x,y) \otimes \cA_2(x',y')\,.$$
\end{definition}

\begin{remark}\label{rk:monoid-str}
The tensor product of dg categories gives rise to a symmetric monoidal structure on $\dgcat$,
with unit object the dg category $\underline{k}$ (see \S\ref{sub:dg-cells}).
Moreover, this model structure is easily seen to be closed. However,
the model structure of Theorem~\ref{thm:Morita} endowed with
this symmetric monoidal structure is {\em not} a symmetric monoidal model category,
as the tensor product of two cofibrant objects in $\dgcat$ is not cofibrant in general;
see \cite[\S 4]{Toen}. Nevertheless, the tensor product can be derived into a bifunctor
\begin{eqnarray*}
-\otimes^{\bbL}-\,: \Hmo \times \Hmo \too \Hmo && (\cA_1, \cA_2)
\mapsto Q(\cA_1)\otimes \cA_2=:\cA_1 \otimes^{\bbL}\cA_2\,,
\end{eqnarray*}
where $Q(\cA_1)$ is a Morita cofibrant resolution of $\cA_1$.
\end{remark}
Now, let $\cB$ and $\cA$ be small dg categories. For any object $x$ in $\cB$,
we have a dg functor $\cA \to \cBo \otimes \cA$. It sends an object $y$ in $\cA$ to $(x,y)$
and for each ordered pair of objects $(y,z)$ in $\cA$, the morphism in $\cC(k)$
$$ \cA(y,z) \too (\cBo \otimes \cA)((x,y),(x,z))=\cBo(x,x) \otimes \cA(y,z)$$
is given by the tensor product of the unit morphism $k \to \cBo(x,x)$
with the identity morphism on $\cA(y,z)$. Take a Morita cofibrant resolution $Q(\cBo)$ of $\cBo$
with the same set of objects; see~\cite[Proposition~2.3(2)]{Toen}. Since $\cBo$ and $Q(\cBo)$
have the same set of objects, one sees that for any object $x$ in $\cB$, we have a dg functor
\begin{equation}\label{eq:restrict}
 i_x: \cA \too Q(\cBo)\otimes\cA =: \cBo \otimes^{\bbL} \cA\,.
\end{equation}
\begin{definition}\label{def:(r)quasicomp}
Let $\cA$ be a small dg category. A $\cA$-module $M$ is called {\em perfect}
if it is compact as an object in the derived category $\cD(\cA)$. We denote by
$\cC_{\dg}(\cA)^{\mathsf{pe}}$ the dg category of perfect $\cA$-modules, and put
$$\perf(\cA)=\cC_{\dg}(\cA)^{\mathsf{pe}}\cap \cD_\dg(\cA)\, .$$
We thus have an equivalence of triangulated categories
$$\dgHo(\perf(\cA))\simeq\cD_c(\cA)\, .$$
Let $\cB$ and $\cA$ be small dg categories. A $(\cB^\op \otimes^{\bbL}\cA)$-module
$X$ is said to be {\em locally perfect over $\cB$} if, for any object $x$ in $\cB$, the $\cA$-module
$i^{\ast}_x(X)$ (see \eqref{eq:restrict} and \S\ref{sub:modules}) is perfect.
We denote by $\cC_{\dg}(\cB^\op \otimes^{\bbL} \cA)^{\mathsf{lpe}}$ the dg category
of locally perfect $(\cB^\op\otimes^\bbL\cA)$-modules over $\cB$,
and put
$$\rep(\cB,\cA)=\cC_{\dg}(\cB^\op \otimes^{\bbL} \cA)^{\mathsf{lpe}}
\cap \cD_\dg(\cB^\op\otimes^\bbL\cA )\, .$$
Note that $\dgHo(\rep(\cB,\cA))$ is canonically equivalent to the full
triangulated subcategory of $\cD(\cB^\op\otimes^{\bbL}\cA)$ spanned by the locally perfect
$(\cB^\op\otimes^\bbL\cA)$-modules over $\cB$.
\end{definition}

\begin{remark}\label{rk:Yoneda}
For any dg category $\cA$, the Yoneda embedding
$\underline{h}:\cA\to\perf(\cA)$ is a derived Morita equivalence, and $\perf(\cA)$
is Morita fibrant. This construction thus provides us a canonical
fibrant replacement functor for the Morita model structure.
\end{remark}

\begin{theorem}[To{\"e}n]\label{thm:Toen0}
Given small dg categories $\cB$ and $\cA$, we have a natural bijection
$$ \Hom_{\Hmo}(\cB, \cA) \simeq
\mathrm{Iso} \, \dgHo(\rep(\cB,\cA))$$
(where $\mathrm{Iso}\, \cC$ stands for the set of isomorphism classes
of objects in $\cC$). Moreover, given small
dg categories $\cA_1$, $\cA_2$, and $\cA_3$, the composition in $\Hmo$ corresponds to the
derived tensor product of bimodules\,:
$$
\begin{array}{ccc}
\mathrm{Iso} \, \dgHo(\rep(\cA_1,\cA_2)) \times
\mathrm{Iso} \, \dgHo(\rep(\cA_2,\cA_3))
&\too& \mathrm{Iso} \, \dgHo(\rep(\cA_1,\cA_3)) \\
 ([X],[Y]) &\longmapsto& [X \otimes_{\cA_2}^{\bbL} Y]\,.
\end{array}$$
\end{theorem}

\begin{proof}
See \cite[Corollary~4.8]{Toen} and \cite[Remarque~5.11]{IMRN}.
\end{proof}


\begin{theorem}[To{\"e}n]\label{thm:Toen}
The derived tensor product $-\otimes^{\bbL}-$ on $\Hmo$ admits
the bifunctor
$$\rep(-,-): \Hmo^{\op} \times \Hmo \too \Hmo$$
as an internal Hom-functor.
\end{theorem}

\begin{proof}
See \cite[Theorem~6.1]{Toen} and \cite[Remarque~5.11]{IMRN}.
\end{proof}

\begin{corollary}\label{cor:Toen}
Given small dg categories $\cA_1$, $\cA_2$ and $\cA_3$, we have a natural isomorphism in $\Ho(\sSet)$
\begin{equation}\label{eq:adj}
\Map(\cA_1 \otimes^{\bbL} \cA_2, \cA_3) \simeq \Map(\cA_1, \rep(\cA_2, \cA_3))\,.
\end{equation}
\end{corollary}

\begin{proof}
See \cite[Corollary~6.4]{Toen} and \cite[Remarque~5.11]{IMRN}.
\end{proof}

\begin{remark}\label{rk:Toen}
Observe that, when $\cB=\underline{k}$, the
dg category $\cB^\op \otimes^{\bbL}\cA$ identifies with $\cA$, which gives a
natural isomorphism in $\Hmo$\,:
\begin{equation}\label{eq:ident}
\rep(\underline{k}, \cA)\simeq\perf(\cA)\,.
\end{equation}
\end{remark}

We finish this section by showing how to endow the derivator associated
to the Morita model structure on dg categories with a symmetric monoidal
structure; see Theorem~\ref{thm:Mdgcat}.
Recall from \S\ref{sub:Monoidstructure} that the Morita model structure is {\em not} a monoidal model structure and so we cannot apply the general Proposition~\ref{prop:Mderivator}.
We sidestep this difficulty by introducing the notion of {\em $k$-flat dg category}.

\begin{definition}\label{def:k-flatdg}
A complex of $k$-modules is called {\em $k$-flat} if for any acyclic complex $N$,
the complex $M\otimes N$ is acyclic.

A dg category is called {\em $k$-flat}
if for each ordered pair of objects $(x,y)$ in $\cA$, the complex $\cA(x,y)$ is $k$-flat.
\end{definition}

\begin{notation}\label{not:k-flat}
We denote by $\dgcat_{\Flat}$ the category of $k$-flat dg categories.
\end{notation}

\begin{remark}\label{rk:k-flatprop}
Clearly $\dgcat_{\Flat}$ is a full subcategory of $\dgcat$. Since complexes of $k$-flat modules
are stable under tensor product, the category $\dgcat_{\Flat}$ is a symmetric monoidal full
subcategory of $\dgcat$ (see \S \ref{sub:Monoidstructure}). Moreover, the tensor product
$$ -\otimes - : \dgcat_{\Flat} \times \dgcat_{\Flat} \too \dgcat_{\Flat}$$
preserves derived Morita equivalences (in both variables).
\end{remark}

\begin{notation}\label{not:Der-Mon}
We denote by $\HO(\dgcat)$ the derivator associated to the Morita model structure on dg categories;
see Theorem~\ref{thm:Morita}.
\end{notation}

\begin{theorem}\label{thm:Mdgcat}
The derivator $\HO(\dgcat)$ carries a symmetric monoidal structure $-\otimes^{\bbL}-$;
see \S\ref{sub:M1}. Moreover, the induced symmetric monoidal structure on $\HO(\dgcat)(e)=\Hmo$
coincides with the one described in Remark~\ref{rk:monoid-str}.
\end{theorem}

\begin{proof}
Let $\HO(\dgcat_\Flat)=\underline{\dgcat_{\Flat}}[{\cM or}^{-1}]$
be the prederivator associated to the class $\cM or$
of derived Morita equivalences in $\dgcat_{\Flat}$; see~\S\ref{sub:prederivators}.
Since $\dgcat_{\Flat}$ is a symmetric monoidal category, and the tensor product preserves the class
$\cM or$, \ie $\cM or \otimes \cM or \subset \cM or$, the prederivator
$\HO(\dgcat_\Flat)$ carries a natural symmetric monoidal structure;
see \S\ref{sub:M1}. Now, choose a functorial Morita cofibrant resolution functor
$$
\begin{array}{lcr}
Q: \dgcat \too \dgcat & & Q \Rightarrow \mbox{Id}\,.
\end{array}
$$ 
Thanks to \cite[Proposition~2.3(3)]{Toen} every cofibrant dg category is $k$-flat, and so we obtain functors
$$
\begin{array}{lcr}
i : \dgcat_{\Flat} \hookrightarrow \dgcat & & Q: \dgcat \too \dgcat_{\Flat}
\end{array}
$$
and natural transformations
\begin{eqnarray}\label{eq:nat}
i \circ Q \Rightarrow \mbox{Id} & & Q \circ i \Rightarrow \mbox{Id}\,.
\end{eqnarray}
Since the functors $i$ and $Q$ preserve derived Morita equivalences and the above natural
transformations (\ref{eq:nat}) are objectwise derived Morita equivalences, we obtain by
$2$-functoriality (see \S\ref{sub:prederivators}) morphisms of derivators
$$
\begin{array}{lcr}
i : \HO(\dgcat_\Flat) \too \HO(\dgcat) & & Q: \HO(\dgcat) \too \HO(\dgcat_\Flat)
\end{array}
$$
which are quasi-inverse to each other.
Using this equivalence of derivators, we transport the
monoidal structure from $\HO(\dgcat_\Flat)$ to $\HO(\dgcat)$.
The fact that the induced monoidal structure on $\HO(\dgcat)(e)=\Hmo$
coincides with the one described in Remark~\ref{rk:monoid-str} is now clear.
\end{proof}

\begin{proposition}\label{prop:Fltdgcat}
The tensor product on $\HO(\dgcat)$ of Theorem~\ref{thm:Mdgcat}
preserves homotopy colimits in each variable.
\end{proposition}

\begin{proof}
This follows immediately from Corollary~\ref{cor:Toen}.
\end{proof}
\section{Homotopically finitely presented dg categories}\label{sub:hfp}
In this section we give a new characterization of homotopically
finitely presented dg categories; see Theorem~\ref{thm:enhance}(3).
Using this new characterization we show that homotopically finitely
presented dg categories are stable under derived tensor product;
see Theorem~\ref{thm:tensorstable}. This stability result will play
a key role in the construction of a symmetric monoidal model structure
on the category of non-commutative motives; see Section~\ref{chap:monoidal}.

\begin{proposition}\label{prop:key}
For any small dg category $\cA$ and $n \in \bbZ$,
we have natural isomorphisms in $\Hmo$
\begin{equation*}
\cS(n)^\op \otimes \cA\simeq \cS(n)^\op \otimes^\bbL \cA \simeq \rep(\dgS(n), \cA) \,.
\end{equation*}
\end{proposition}
\begin{proof}
Since $\dgS(n) \simeq \dgS(n)^\op$ is cofibrant, we obtain the following isomorphism in $\Hmo$
$$ \dgS(n)^\op \otimes^\bbL \cA\simeq \dgS(n)^\op \otimes \cA\, .$$
By construction, the dg category $\dgS(n)$ is also {\em locally perfect},
\ie its complexes of morphisms are perfect complexes of $k$-modules;
see~\cite[Definition~2.4(1)]{Moduli}. We obtain then by \cite[Lemma\,2.8]{Moduli}
the following inclusion
$$\dgS(n)^\op\otimes \cA\simeq \perf(\dgS(n)^\op\otimes^\bbL \cA) \subset \rep(\dgS(n), \cA) \,.$$
Let us now prove the converse inclusion. Observe that we have two natural dg functors
$$
\begin{array}{lccr}
\begin{array}{rcc}
i_1: \cA & \too & \dgS(n)^\op\otimes \cA \\
x & \longmapsto & (1,x)
\end{array}
& &
\begin{array}{rcc}
i_2: \cA & \too & \dgS(n)^\op\otimes \cA \\
x & \longmapsto & (2,x)
\end{array}\,,
\end{array}
$$
and the category of $(\dgS(n)^\op\otimes \cA)$-modules identifies
with the category of morphisms of degree $n$ in $\cC(\cA)$, \ie to the category
of triples $(M,M',f)$, where $M$ and $M'$ are $\cA$-modules, and
$f:M\too M'[n]$ is a morphism of $\cA$-modules. Under this identification, we 
obtain the following extension of scalars functors (see~\S\ref{sub:modules})\,:
\begin{eqnarray*}
(i_1)_!: \cC(\cA) \too \cC(\dgS(n)^\op\otimes \cA) && M \longmapsto (M,M[-n],1_M) \\
(i_2)_!: \cC(\cA) \too \cC(\dgS(n)^\op\otimes \cA) && M \longmapsto (0,M,0)\,.
\end{eqnarray*}
Now, let $X$ be an object in $\rep(\dgS(n), \cA)$.
Note that $X$ corresponds to a cofibration
$$
\xymatrix{
*+<2.5ex>{f: M} \ar@{>->}[r] & M'[n]
}
$$
in $\cC(\cA)$ between cofibrant and perfect $\cA$-modules (see Definition~\ref{def:(r)quasicomp}).
Consider the following short exact sequence of morphisms of $\cA$-modules
\begin{equation}\label{eq:ses} 
\xymatrix{
*+<2.5ex>{(M[-n])[n]=M} \ar@{>->}[rr]^-f && M'[n] \ar[rr] && \mathsf{coker}(f)=(\mathsf{coker}(f)[-n])[n] \\
M \ar@{=}[rr]  \ar[u]^{1_M} && *+<2.5ex>{M} \ar@{>->}[u]_f \ar[rr] && 0 \ar[u]
}
\end{equation}
where $\mathsf{coker}(f)$ denotes the cokernel of $f$ in the category $\cC(\cA)$.
Since $M$ and $M'[n]$ are cofibrant and perfect $\cA$-modules, and $f$ is a cofibration, the
$\cA$-module $\mathsf{coker}(f)$ is also cofibrant and perfect. Perfect modules are stable
under suspension, and so $\mathsf{coker}(f)[-n]$ is a perfect (and cofibrant) $\cA$-module.
We have natural isomorphisms\,:
$$
\begin{array}{lcr}
 (i_1)_!(M)\simeq  (M,M[-n],1_M)  &
 \mathrm{and} &
 (i_2)_!(\mathsf{coker}(f)[-n])\simeq (0,\mathsf{coker}(f)[-n],0) \,.
\end{array}
$$
Since the extension of scalars functors $\bbL (i_1)_!$ and $\bbL (i_2)_!$
preserve perfect modules, these two objects are perfect.
Finally, since $\perf(\dgS(n)^\op\otimes \cA)$ is stable under extensions
in $\cD_\dg(\dgS(n)^\op\otimes \cA)$, and in the above short exact
sequence (\ref{eq:ses}) the left and right vertical morphisms belong to $\perf(\dgS(n)^\op\otimes \cA)$,
we conclude that the object $X$ also belongs to $\perf(\dgS(n)^\op\otimes \cA)$.
\end{proof}
\begin{definition}[{\cite[Definition~2.1(3)]{Moduli}}]\label{def:HFP}
Let $\cM$ be a Quillen model category. An object $X$ in $\cM$ is
called {\em homotopically finitely presented} if for any filtered system
of objects $\{Y_j\}_{j \in J}$ in $\cM$, the induced map
\begin{equation*}
\underset{j \in J}{\hocolim}\, \Map(X,Y_j) \too \Map(X,\underset{j \in J}{\hocolim}\, Y_j)
\end{equation*}
is an isomorphism in $\Ho(\sSet)$. 
\end{definition}
\begin{theorem}\label{thm:enhance}
Let $\cB$ be a small dg category. In the Morita model structure (see Theorem~\ref{thm:Morita})
the following conditions are equivalent\,:
\begin{itemize}
\item[(1)] The dg category $\cB$ is homotopically finitely presented;
\item[(2)] The dg category $\cB$ is a retract in $\Hmo$ of a finite dg cell (see Definition~\ref{def:dg-cell});
\item[(3)] For any filtered system $\{\cA_j\}_{j \in J}$ of dg categories, the induced morphism
\begin{equation}\label{eq:induced}
\underset{j \in J}{\hocolim}\, \rep(\cB,\cA_j) \stackrel{\sim}{\too} \rep(\cB,\underset{j \in J}{\hocolim}\, \cA_j)
\end{equation}
is an isomorphism in $\Hmo$.
\end{itemize}
\end{theorem}
\begin{proof}
$(1) \Leftrightarrow (2)$ : This equivalence follows from \cite[Proposition~5.2]{Duke} and
\cite[Example~5.1]{Duke}. 

\vspace{0.2cm}

$(3) \Rightarrow (1):$ Let $\{ \cA_j \}_{j \in J}$ be a filtered system of dg categories. By
hypothesis, we have an induced isomorphism in $\Hmo$
\begin{equation}\label{eq:iso} 
\underset{j \in J}{\hocolim}\, \rep(\cB,\cA_j) \stackrel{\sim}{\too} \rep(\cB,\underset{j \in J}{\hocolim}\, \cA_j)\,.
\end{equation}
Thanks to Corollary~\ref{cor:Toen} we have, for any dg category $\cA$, natural isomorphisms in $\Ho(\sSet)$\,:
$$ \Map(\cB, \cA) \simeq \Map(\underline{k} \otimes^{\bbL} \cB, \cA)
\simeq \Map(\underline{k}, \rep(\cB, \cA))\,.$$
The dg category $\underline{k}$ is a finite dg cell, and so using equivalence
$(1) \Leftrightarrow (2)$ we conclude that $\underline{k}$ is homotopically finitely presented.
Therefore, by applying the functor
$$ \Map(\underline{k}, -) : \Hmo \too \Ho(\mathsf{sSet})$$
to the above isomorphism (\ref{eq:iso}), we obtain an induced isomorphism in $\Ho(\sSet)$
$$\underset{j \in J}{\hocolim}\, \Map(\cB,\cA_j) \stackrel{\sim}{\too}
\Map(\cB,\underset{j \in J}{\hocolim}\, \cA_j)\,.$$
This shows that $\cB$ is homotopically finitely presented.

\vspace{0.2cm}

$(2) \Rightarrow (3):$ The class of objects in $\Hmo$ which satisfy condition (3) is clearly
stable under retracts. Moreover, given a small dg category $\cA$, the functor 
$$\rep(-,\cA): \Hmo^\op \too \Hmo$$ sends homotopy colimits to homotopy limits.
Since by construction homotopy pullbacks in $\Hmo$ commute with filtered homotopy
colimits, we conclude that the class of objects in $\Hmo$ which satisfy condition (3) is
also stable under homotopy pushouts. Therefore, it is enough to verify condition (3) for
the domains and codomains of the elements of the set $I$; see Notation~\ref{not:I-cell}.
If $\cB = \emptyset$ this is clear. The case where $\cB =\underline{k}$ was proved in
\cite[Lemma\,2.10]{Moduli}. If $\cB= \dgD(n)$, $n \in \bbZ$, we have a natural derived Morita equivalence
$$ \underline{k} \amalg^{\bbL}_{\emptyset} \underline{k} \simeq \underline{k}\amalg \underline{k}
\stackrel{\sim}{\too} \dgD(n)\,.$$
Therefore, it remains to prove the case where $\cB= \dgS(n)$, $n \in \bbZ$. Since by Proposition~\ref{prop:Fltdgcat} the derived tensor product preserves homotopy colimits, this case 
follows automatically from Proposition~\ref{prop:key}.
\end{proof}

\begin{theorem}\label{thm:tensorstable}
Let $\cB_1$ and $\cB_2$ be homotopically finitely presented dg categories.
Then $\cB_1 \otimes^{\mathbb{L}}\cB_2$ is a homotopically finitely presented dg category.
\end{theorem}
\begin{proof}
Let $\{\cA_j\}_{j \in J}$ be a filtered system of dg categories. The proof is a
consequence of the following weak equivalences\,:
\begin{eqnarray}
\Map(\cB_1 \otimes^{\mathbb{L}}\cB_2, \underset{j \in J}{\hocolim}\, \cA_j) & \simeq & 
\Map(\cB_1, \rep(\cB_2, \underset{j \in J}{\hocolim}\,\cA_j))  \label{eq:number1} \\
& \simeq & \Map(\cB_1, \underset{j \in J}{\hocolim}\, \rep(\cB_2,\cA_j)) \label{eq:number2} \\
& \simeq & \underset{j \in J}{\hocolim}\, \Map(\cB_1, \rep(\cB_2,\cA_j)) \label{eq:number3} \\
& \simeq &
\underset{j \in J}{\hocolim}\,\Map(\cB_1 \otimes^{\mathbb{L}}\cB_2, \cA_j)\,. \label{eq:number4} \\ \nonumber
\end{eqnarray} 
Equivalences (\ref{eq:number1}) and (\ref{eq:number4}) follow from Corollary~\ref{cor:Toen},
Equivalence (\ref{eq:number2}) follows from Theorem~\ref{thm:enhance}, and
Equivalence (\ref{eq:number3}) follows from Definition~\ref{def:HFP}.  
\end{proof}

\section{Saturated dg categories}\label{chap:saturated}
In this section we introduce Kontsevich's notion of saturated dg category.
Following To{\"e}n's work, we characterize these dg categories as the dualizable objects
in the Morita homotopy category; see Theorem~\ref{thm:Toen-dualizable}.
This conceptual characterization will play an important role in our applications;
see Section~\ref{sub:applications}.
\begin{definition}[Kontsevich]{(\cite{IAS, Lattice, KS})}\label{def:saturated}
\begin{itemize}
\item[-] A small dg category $\cA$ is called {\em smooth} if the right
dg $(\cA^\op\otimes^{\bbL}\cA)$-module
\begin{eqnarray}\label{eq:Kont}
\cA(-,-): \cA\otimes^{\bbL}\cA^\op \too \cC_{\dg}(k) && (x,y) \mapsto \cA(y,x)
\end{eqnarray}
is perfect; see Definition~\ref{def:(r)quasicomp}.
\item[-] A small dg category $\cA$ is called {\em proper} if for each ordered pair
of objects $(x,y)$ in $\cA$, the complex of $k$-modules $\cA(x,y)$ is perfect.
\item[-] A small dg category $\cA$ is called {\em saturated} if it is smooth and proper.
\end{itemize}
\end{definition}

\begin{remark}\label{rk:bi-module}
Given a dg functor $F:\cB \to \cA$, we have a natural $(\cB^\op\otimes^{\bbL}\cA)$-module
\begin{eqnarray*}
\cB\otimes^{\bbL}\cA^\op \too \cC_{\dg}(k) && (x,y) \mapsto \cA(y,Fx)\,,
\end{eqnarray*}
which belongs to $\rep(\cB,\cA)$ (see Definition~\ref{def:(r)quasicomp}). The above $(\cA^\op\otimes^{\bbL}\cA)$-module $\cA(-,-)$ (\ref{eq:Kont}) is obtained by taking $\cB=\cA$ and $F$ the identity dg functor.
\end{remark}

\begin{notation}\label{not:sat}
Note that the class of smooth (resp. proper) dg categories
is invariant under derived Morita equivalences (see Definition~\ref{def:Morita}).
We denote by $\Hmo_{\mathsf{sat}}$ the full subcategory of $\Hmo$
(see Notation~\ref{not:hom-Morita}) whose objects are the saturated dg categories.
\end{notation}

\begin{example}
Consider the dg categories $\cS(n), n \geq 0$, from \S\ref{sub:dg-cells}(ii).
By construction these dg categories are proper. Thanks to Proposition~\ref{prop:key},
we have a natural isomorphism $\cS(n)^\op\otimes^{\bbL}\cS(n) \simeq \rep(\cS(n), \cS(n))$
in $\Hmo$. Since the $(\cS(n)^\op \otimes^{\bbL}\cS(n))$-module $\cS(n)(-,-)$ belongs to
$\rep(\cS(n),\cS(n))$ (see Remark~\ref{rk:bi-module}) and we have a natural derived Morita
equivalence $\cS(n)^\op\otimes^{\bbL}\cS(n)\to \perf(\cS(n)^\op\otimes^{\bbL}\cS(n))$
(see Remark~\ref{rk:Yoneda}), we conclude that the dg categories $\cS(n)$ are also smooth.
Therefore, they are saturated.
\end{example}

\begin{example}\label{ex:Toen}
\begin{itemize}
\item[(i)] Let $X$ be a quasi-compact and separated $k$-scheme. Consider the category
$\cC(QCoh(X))$ of (unbounded) complexes of quasi-coherent sheaves on $X$.
Thanks to~\cite{HoveySchemes}, $\cC(QCoh(X))$ is a model category with
monomorphisms as cofibrations and quasi-isomorphisms as weak equivalences.
Moreover, when endowed with its natural $\cC(k)$-enrichment, $\cC(QCoh(X))$ becomes
a $\cC(k)$-model category in the sense of~\cite[Definition~4.2.18]{Hovey}.
Let $\cL_{qcoh}(X)$ be the dg category of fibrant objects in $\cC(QCoh(X))$.
Note that $\dgHo(\cL_{qcoh}(X))$ is naturally
equivalent to the (unbounded) derived category $\cD_{qcoh}(X)$ of quasi-coherent
sheaves on $X$. Finally, let $\perf(X)$ be the full dg subcategory of $\cL_{qcoh}(X)$
whose objects are the perfect complexes. Note that $\dgHo(\perf(X))$ is naturally equivalent
to the category of compact objects in $\cD_{qcoh}(X)$.
Thanks to To{\"e}n (see~\cite[Lemma~3.27]{Moduli}), when $X$ is a smooth and
proper $k$-scheme, $\perf(X)$ is a saturated dg category.
\item[(ii)] Let $A$ be a $k$-algebra, which is projective of finite rank as a $k$-module, and
of finite global cohomological dimension. Then, the dg category of perfect complexes of
$A$-modules is a saturated dg category.
\item[(iii)] For examples coming from {\em deformation quantization}, we invite the reader
to consult \cite{IAS}. 
\end{itemize}
\end{example}

\begin{definition}\label{def:dualizable}
Let $\cC$ be a symmetric monoidal category with monoidal product $\otimes$ and unit object ${\bf 1}$.
An object $X$ in $\cC$ is called {\em dualizable} (or {\em rigid}) if there exists an object
$X^{\vee}$ in $\cC$, and maps $\mathrm{ev}:X \otimes X^{\vee} \to {\bf 1}$ and
$\delta: {\bf 1} \to X^{\vee} \otimes X$ such that the composites
\begin{equation}\label{eq:dualizable1}
X \simeq  X \otimes{\bf 1} \stackrel{\id \otimes \delta}{\too}
X \otimes X^{\vee} \otimes X \stackrel{  \mathrm{ev} \otimes \id}{\too}
{\bf 1} \otimes  X  \simeq X
\end{equation}
and 
\begin{equation}\label{eq:dualizable2}
 X^{\vee} \simeq  {\bf 1} \otimes X^{\vee}  \stackrel{\delta \otimes \id}{\too}
 X^{\vee} \otimes X \otimes X^{\vee} \stackrel{\id \otimes \mathrm{ev}}{\too}
  X^{\vee} \otimes {\bf 1} \simeq X^{\vee}
\end{equation}
are identities. The object $X^{\vee}$ is called the {\em dual} of $X$, the map $\mathrm{ev}$
is called the {\em evaluation map}, and the map $\delta$ is called the {\em co-evaluation map}.
\end{definition}
\begin{remark}\label{rk:adj-dualizable}
\begin{itemize}
\item[(i)] Given $((X^{\vee})_1, \mathrm{ev}_1, \delta_1)$ and $((X^{\vee})_2, \mathrm{ev}_2, \delta_2)$
as in Definition~\ref{def:dualizable}, there is a unique isomorphism
$(X^{\vee})_1 \stackrel{\sim}{\to} (X^{\vee})_2$ making the natural diagrams commute.
Therefore, the dual of $X$, together with the evaluation and the co-evaluation maps, is well-defined
up to unique isomorphism.
\item[(ii)] Let $X$ be a dualizable object in $\cC$. Thanks to Equations (\ref{eq:dualizable1})
and (\ref{eq:dualizable2}), the evaluation and co-evaluation maps give rise to an adjunction\,:
\begin{equation}\label{eq:adj-dualizable}
\xymatrix{
\cC \ar@<1ex>[d]^{X^{\vee}\otimes -} \\
\cC \ar@<1ex>[u]^{- \otimes X}
}
\end{equation}
In fact, an object $X$ of $\cC$ is dualizable if and only if there exists
an object $X'$, together with a functorial isomorphism
\begin{equation*}
\Hom_\cC(Y\otimes X,Z)\simeq \Hom_\cC(Y,X'\otimes Z)
\end{equation*}
for any objects $Y$ and $Z$ in $\cC$ (and of course such an $X'$
is a dual of $X$).
In other words, $X$ is dualizable with dual $X^\vee$ if and only if, for any
object $Z$ of $\cC$, $X^\vee\otimes Z$ is the internal Hom object from
$X$ to $Z$.
\item[(iii)] Let $X$ and $Y$ be dualizable objects in $\cC$. Then $X \otimes Y$ is also a dualizable
object with dual $Y^{\vee} \otimes X^{\vee}$.
\item[(iv)] Let $F:\cC \to \cC'$ be a symmetric monoidal functor. Then, if $X$ is a dualizable object
in $\cC$, $F(X)$ is a dualizable object in $\cC'$ with dual $F(X^{\vee})$.
\end{itemize}
\end{remark}

We are now ready to state the following folklore result:

\begin{theorem}
\label{thm:Toen-dualizable}
The dualizable objects in the Morita homotopy category $\Hmo$ (see \S\ref{sub:Monoidstructure})
are the saturated dg categories. Moreover, the dual of a saturated dg category $\cA$ is its opposite
dg category $\cA^\op$.
\end{theorem}
\begin{proof}
Let us start by introducing the category $\mathsf{DGCAT}$. Its objects are the small dg categories.
Given small dg categories $\cB$ and $\cA$, the set of morphisms from $\cB$ to $\cA$ is the set of
isomorphism classes in $\cD(\cB^\op \otimes^{\bbL}\cA)$.
Given small dg categories $\cA_1$, $\cA_2$, and $\cA_3$, the composition corresponds
to the derived tensor product of bimodules\,:
$$
\begin{array}{ccc}
\mathrm{Iso} \,\cD(\cA_1^\op \otimes^{\bbL} \cA_2)
\times \mathrm{Iso} \,\cD(\cA_2^\op \otimes^{\bbL} \cA_3)
&\too& \mathrm{Iso} \,\cD(\cA_1^\op \otimes^{\bbL} \cA_3) \\
 ([X],[Y]) &\mapsto& [X \otimes_{\cA_2}^{\bbL} Y]\,.
\end{array}
$$
As in the case of $\Hmo$, the derived tensor product of dg categories gives rise to a symmetric
monoidal structure on $\mathsf{DGCAT}$, with unit object the dg category $\underline{k}$.
The key property of $\mathsf{DGCAT}$ is that all its objects are dualizable: given a small dg
category $\cA$, take for dual its opposite dg category $\cAo$, for evaluation map the isomorphism
class in $\cD((\cA\otimes^{\bbL}\cA^\op)^\op\otimes^{\bbL} k) \simeq \cD(\cA^\op\otimes \cA)$ of
the $(\cA^\op\otimes^{\bbL}\cA)$-module
\begin{eqnarray}\label{eq:bi-module1}
\cA(-,-): \cA\otimes^{\bbL}\cA^\op \too \cC_{\dg}(k) && (x,y) \mapsto \cA(y,x)\,,
\end{eqnarray}
and for co-evaluation map the isomorphism class in
$\cD(\underline{k}^\op\otimes^{\bbL}(\cA^\op\otimes^{\bbL}\cA)) \simeq \cD(\cA^\op\otimes \cA)$
of the same $(\cA^\op\otimes^{\bbL}\cA)$-module (\ref{eq:bi-module1}).
With this choices, both composites (\ref{eq:dualizable1}) and (\ref{eq:dualizable2}) are identities.
Note also that, by construction of $\mathsf{DGCAT}$, and by Theorem~\ref{thm:Toen0}, we have a
natural symmetric monoidal functor
\begin{eqnarray}\label{eq:functor}
\Hmo \too \mathsf{DGCAT} && \cA \mapsto \cA\,,
\end{eqnarray}
which is faithful but not full.

Let $\cA$ be a dualizable object in $\Hmo$, with dual $\cA^{\vee}$, evaluation map
$\mathrm{ev}: \cA \otimes^{\bbL}\cA^{\vee} \to \underline{k}$ and co-evaluation map
$\delta: \underline{k} \to \cA^{\vee} \otimes^{\bbL}\cA$. Since the above functor
(\ref{eq:functor}) is symmetric monoidal, Remark~\ref{rk:adj-dualizable}(iv) implies that $\cA$
is dualizable in $\mathsf{DGCAT}$. By unicity of duals (see Remark~\ref{rk:adj-dualizable}(i)),
$\cA^{\vee}$ is the opposite dg category $\cA^\op$, and $\mathrm{ev}$ and $\delta$ are the isomorphism class in
$\cD(\cA^\op\otimes^{\bbL}\cA)$ of the above $(\cA^\op\otimes^{\bbL}\cA)$-module $\cA(-,-)$
(see \ref{eq:bi-module1}). The
morphism $\mathrm{ev}$ belongs to $\Hmo$, and so
$\cA(-,-)$ takes values in perfect complexes of $k$-modules.
We conclude then that $\cA$ is proper. Similarly, since $\delta$ is a
morphism in $\Hmo$, $\cA(-,-)$ belongs to
$\perf(\cA^\op\otimes^\bbL\cA)$. In
this case we conclude that $\cA$ is smooth. In sum, we have shown that $\cA$ is saturated
and that $\cA^\vee=\cA^\op$.

Now, let $\cA$ be a saturated dg category. Since $\cA$ is proper $\cA(-,-)$
takes values in perfect complexes of $k$-modules. Similarly, since $\cA$ is smooth $\cA(-,-)$ belongs to $\perf(\cA^\op\otimes^\bbL\cA)$.
We conclude that the evaluation and co-evaluation maps
of $\cA$ in $\mathsf{DGCAT}$ belong to the subcategory $\Hmo$.
This implies that $\cA$ is dualizable in $\Hmo$.
\end{proof}
We finish this section with a comparison between saturated and homotopically finitely presented
dg categories (see \S\ref{sub:hfp}).

\begin{lemma}\label{lem:tensor-rep}
Let $\cB$ be a saturated dg category. Then, for every dg category $\cA$,
we have a functorial isomorphism in $\Hmo$
$$ \cB^\op \otimes^{\bbL}\cA \simeq \rep(\cB,\cA)\,.$$
\end{lemma}

\begin{proof}
Since $\cB^\vee=\cB^\op$, the general theory of dualizable objects implies the claim; see Remark \ref{rk:adj-dualizable}~(iii).
\end{proof}

\begin{proposition}
Every saturated dg category is homotopically finitely presented.
\end{proposition}

\begin{proof}
Let $\cB$ be a saturated dg category.
Since the derived tensor product of dg categories preserves homotopy colimits in each variable
(see Proposition~\ref{prop:Fltdgcat}), Lemma~\ref{lem:tensor-rep} implies that
the functor $\rep(\cB,-)$ commutes with homotopy colimits.
Using Theorem \ref{thm:enhance}(3) the proof is achieved.
\end{proof}

\section{Simplicial presheaves}\label{chap:simplicial}
In this section we present a general theory of symmetric monoidal
structures on simplicial presheaves. These
general results will be used in the construction of an explicit symmetric monoidal structure on the category
of non-commutative motives; see Section~\ref{chap:monoidal}. 

Starting from a small symmetric monoidal category $\cC$, we recall that the tensor product of $\cC$ extends to simplicial presheaves on $\cC$; see \S\ref{sub:Day}. Then, we prove that this symmetric monoidal structure is compatible
with the projective model structure; see Theorem~\ref{thm:Mprojective}.
 Finally, we study its behavior under left Bousfield localizations. We describe
 in particular a ``minimal'' compatibility condition between the tensor product and a
 localizing set; see Theorem~\ref{thm:monloc}.

\subsection{Notations}\label{sub:not-simpl} Throughout this section
$\cC$ will denote a (fixed) small category.
\begin{itemize}
\item[(i)] We denote by $\pref\cC$ the {\em category of presheaves of sets on $\cC$},
\ie the category of contravariant functors from $\cC$ to $\Set$.
\item[(ii)] Given an object $\alpha$ in $\cC$, we still denote by $\alpha$
the presheaf represented by $\alpha$
\begin{eqnarray*}
{\alpha}: \cC^\op \too \Set \,,&& \beta \mapsto \Hom_{\cC}(\beta, \alpha)
\end{eqnarray*}
(\ie we consider the Yoneda embedding as a full inclusion).
\item[(iii)] Let $\Delta$ be the category of simplices, \ie the full
subcategory of the category of ordered sets spanned by the
sets $\Delta[n]=\{0,\ldots,n\}$ for $n\geq 0$.
We set
$$\sSet=\pref\Delta\, .$$
\item[(iv)]  We denote by $\spref\cC\simeq\pref{\Delta\times\cC}$
the category of simplicial objects in $\pref\cC$,
\ie the category of contravariant functors from $\Delta$ to $\pref\cC$.
\item[(v)] Finally, we consider $\pref\cC$ as a full subcategory of $\spref\cC$.
A presheaf of sets on $\cC$ is identified with a simplicially constant object of $\spref\cC$.
\end{itemize}
\subsection{Simplicial structure}\label{sub:simpl}
Recall that we have a bifunctor
$$-\otimes- : \pref\cC\times\Set\too \pref\cC$$
defined by
$$X\otimes K=\coprod_{k\in K} X\, .$$
This defines an action of the category $\Set$ on $\pref\cC$.
This construction extends to simplicial objects
\begin{eqnarray*}
\spref\cC \times  \sSet \too  \spref\cC && (F, K) \mapsto F \otimes K\,,
\end{eqnarray*}
where for $n\geq 0$\,:
$$ (F\otimes K)_n =F_n\otimes K_n \,.$$
This makes $\spref\cC$ into a simplicial category; see for instance \cite[\S II Definition~2.1]{Jardine}.
\subsection{Quillen Model structure}\label{sub:modelBK}
The generating cofibrations of the classical cofibrantly generated Quillen model structure on
$\sSet$ are the boundary inclusions
\begin{eqnarray*} 
i_n: \partial \Delta[n] \too \Delta[n] && n \geq 0
\end{eqnarray*}
and the generating trivial cofibrations are the horn inclusions
\begin{eqnarray*} j^k_n: \Lambda[k,n] \too \Delta[n] && 0 \leq k \leq n,\,\,n \geq 1\,.
\end{eqnarray*}
We have the \emph{projective model structure} on $\spref\cC$: the weak equivalences
are the termwise simplicial weak equivalences, and the fibrations are the
termwise Kan fibrations; see for instance \cite[page~314]{BK}
or \cite[Theorem~11.6.1]{Hirschhorn}. 
The projective model structure is proper and cellular/combinatorial.
In particular, it is cofibrantly generated, with
generating cofibrations
$$1_\alpha \otimes i_n: \alpha\otimes \partial \Delta[n] \too
\alpha\otimes \Delta[n] \ , \quad \alpha \in \cC,\,\, n \geq 0\,,$$
and generating trivial cofibrations
$$1_\alpha \otimes j^k_n: \alpha\otimes \Lambda[k,n]
\too \alpha \otimes \Delta[n]\ , \quad \alpha \in \cC,\,\, 0 \leq k \leq n, \,\, n \geq 1\,.$$
In particular, observe that representable presheaves are cofibrant in $\spref\cC$.
\subsection{Day's convolution product}\label{sub:Day}
Throughout this subsection, and until the end of Section~\ref{chap:simplicial},
we will assume that our (fixed) small category $\cC$ carries a symmetric
monoidal structure, with tensor product $\otimes$ and unit object ${\bf 1}$.
Under this assumption, the general theory of left Kan extensions in categories of presheaves
implies formally that $\pref\cC$ is endowed with a unique closed
symmetric monoidal structure which makes the Yoneda embedding a symmetric monoidal functor.
We will also denote by $\otimes$ the corresponding tensor product on
$\pref\cC$; the reader who enjoys explicit formulas is invited to consult~\cite[\S3]{Day}.

This monoidal structure extends to the category $\spref\cC$ in an obvious way:
given two simplicial presheaves $F$ and $G$ on $\cC$, we define $F\otimes G$
by the formula
$$(F\otimes G)_n=F_n\otimes G_n\ , \quad n\geq 0\, .$$
The functor
$$\sSet\too \spref\cC \ , \quad K\longmapsto {\bf 1}\otimes K$$
is naturally endowed with a structure of symmetric monoidal functor
(where $\sSet$ is considered as a symmetric monoidal category with the
cartesian product as tensor product).

\begin{definition}[{\cite[Definition~4.2.1]{Hovey}} ]\label{def:push-prod}
Given maps $f:X \to Y$ and $g: A \to B$ in a symmetric monoidal category
(with tensor product $\otimes$), the {\em pushout product map} $f\square g$ of $f$ and $g$ is given by\,:
$$ f \square g:  X\otimes B \underset{X\otimes A}{\coprod} Y\otimes A \too Y\otimes B\,.$$
\end{definition}

\begin{theorem}\label{thm:Mprojective}
The category $\spref\cC$, endowed with the projective model category structure
is a symmetric monoidal model category (see \cite[Definition~4.2.6]{Hovey}).
\end{theorem}

\begin{proof}
As the model category of simplicial sets is a symmetric monoidal model category
(with the cartesian product as tensor product),
the result follows from the explicit description of the generating
cofibrations and generating trivial cofibrations of $\spref\cC$: given two
objects $\alpha$ and $\alpha'$ in $\cC$ and two
maps $i:K\to L$ and $i':K'\to L'$ in $\sSet$, we have
$$(1_\alpha\otimes i)\square (1_{\alpha'}\otimes i')\simeq
1_{\alpha\otimes\alpha'}\otimes (i\square i')\,.$$
Since for any object $\alpha$ in $\cC$, the functor
$K\mapsto  \alpha \otimes K$ is a left Quillen functor from $\sSet$ to
$\spref\cC$, the proof is finished.
\end{proof}

Let $\spref\cC_\bullet$ be the category of pointed simplicial presheaves on $\cC$.
The forgetful functor $U:\spref\cC_\bullet\to\spref\cC$ has a left adjoint
\begin{equation}\label{adpointssimpl}
\spref\cC\too\spref\cC_\bullet \ , \quad F\longmapsto F_+\, ,
\end{equation}
where $F_+$ denotes the pointed simplicial presheaf $F\amalg \star$, with $\star$
the terminal object of $\spref\cC$. The category $\spref\cC_\bullet$ is then
canonically a cofibrantly generated model category, in such a way that
the functor \eqref{adpointssimpl} is a left Quillen functor; see \cite[Proposition~1.1.8 and Lemma~2.1.21]{Hovey}.

Furthermore, there is a unique symmetric monoidal structure on $\spref\cC_\bullet$ making
the functor \eqref{adpointssimpl} symmetric monoidal.
The unit object is ${\bf 1}_+$, and the tensor product $\otimes_\bullet$
is defined as follows: for two pointed simplicial presheaves $F$ and $G$,
their tensor product is given by the following pushout in the category
$\spref\cC$ of unpointed simplicial presheaves\,:
\begin{equation*}\begin{split}
\xymatrix{(F\otimes\star) \amalg (\star \otimes G) \ar[r]\ar[d] & F\otimes G \ar[d]\\
\star \ar[r] & F \otimes_\bullet  G}
\end{split}\end{equation*}
In particular, for two simplicial presheaves $F$ and $G$, we have
\begin{equation}\label{adpointmonoidal}
(F\otimes G)_+ \simeq F_+ \otimes_\bullet  G_+\, .
\end{equation}

\begin{proposition}\label{prop:Mprojectivepointed}
With the above definition, the model category $\spref\cC_\bullet$
is a symmetric monoidal model category.
\end{proposition}

\begin{proof}
The generating cofibrations (resp. generating trivial cofibrations) of $\spref\cC_\bullet$
are the maps of shape $A_+\to B_+$ for $A\to B$ a generating
cofibration (resp. generating trivial cofibration) of $\spref\cC$.
As the functor \eqref{adpointssimpl} is a left Quillen functor,
the result follows immediately from
Formula \eqref{adpointmonoidal} and from Theorem \ref{thm:Mprojective}.
\end{proof}

\begin{remark}
In practice, we shall refer to the pointed
tensor product $\otimes_\bullet$ as the canonical tensor product
on $\spref\cC_\bullet$ associated to the monoidal structure
on $\cC$. Whenever there is no ambiguity, we denote the pointed tensor product simply by $\otimes$.
\end{remark}

\subsection{Left Bousfield localization}
\begin{definition}{(\cite[Definition~3.1.4]{Hirschhorn})}
Let $\cM$ be a model category and $S$ a set of morphisms in $\cM$. An
object $X$ in $\cM$ is called {\em $S$-local} if it is fibrant and for every element
$s: A \to B$ of the set $S$, the induced map of homotopy function complexes
$$ s^{\ast}: \Map(B, X) \too \Map(A,X)$$
is a weak equivalence. A map $g:X \to Y$ in $\cM$ is called a {\em $S$-local equivalence}
if for every $S$-local object $W$, the induced map of homotopy function complexes
$$ g^{\ast}: \Map(Y, W) \too \Map(X,W)$$
is a weak equivalence

Recall that, if $\cM$ is cellular (or combinatorial) and left proper, the \emph{left Bousfield localization} of $\cM$ is the model category $\Loc_S\cM$
whose underlying category is $\cM$, whose cofibrations are those of
$\cM$, and whose weak equivalences are the $S$-local weak equivalences; see \cite[Definition~9.3.1(1)]{Hirschhorn}.
The fibrant objects of $\Loc_S\cM$ are then the objects of $\cM$ which are both fibrant and $S$-local, and
the fibrations between fibrant objects in $\Loc_S\cM$ are the fibrations of $\cM$.
\end{definition}

\begin{proposition}\label{prop:loc-mon}
Let $\cM$ be a left proper, cellular (or combinatorial), symmetric monoidal model category
(with tensor product $\otimes$), and $S$ a set of morphisms in $\cM$.
Assume that the following conditions hold\,:
\begin{itemize}
\item[(i)] $\cM$ admits generating sets of cofibrations and of trivial cofibrations consisting of maps between cofibrant objects;
\item[(ii)] every element of $S$ is a map between cofibrant objects;
\item[(iii)] given a cofibrant object $X$, the functor $X\otimes (-)$
sends the element of $S$ to $S$-local weak equivalences.
\item[(iv)] the unit object of the monoidal structure on $\cM$ is cofibrant.
\end{itemize} 
Then the left Bousfield localization $\Loc_S \cM$ of $\cM$
with respect to the set $S$ is a symmetric monoidal model category.
\end{proposition}

\begin{proof}
The left Bousfield localization $\Loc_S \cM$ of $\cM$ with respect to the set
$S$ is cofibrantly generated (see \cite[Theorem~4.1.1(3)]{Hirschhorn}) and
thanks to condition (iv) the unit object is cofibrant. Therefore, by \cite[Lemma\,4.2.7]{Hovey}
it is enough to verify the pushout product axiom on the sets of generating (trivial) cofibrations.
The class of cofibrations in $\Loc_S\cM$ and in $\cM$ is the same, and so half of the pushout product axiom is automatically verified. Now, let
$ 
\xymatrix{
g: A\,\, \ar@{>->}[r]  & B
}
$
be a generating cofibration in $\Loc_S \cM$ and 
$
\xymatrix{
f: X\,\, \ar@{>->}[r]^{\sim} &Y 
}
$
a generating trivial cofibration in $\Loc_S \cM$.
By condition (i), we may assume,
that the objects $X$, $Y$, $A$ and $B$ are
cofibrant. Moreover, condition (iii) implies that tensoring by
a cofibrant object preserves $S$-local weak equivalences.
Consider the following commutative diagram\,:
$$
\xymatrix@=6pt@C=12pt{
*+<2pc>{A \otimes X} \ar@{>->}[dd]_{g \otimes \id_X} \ar[rr]^{ \id_A \otimes f}
& & *+<2pc>{A \otimes Y} \ar@{>->}[dd]^{g \otimes \id_Y} \ar[ld]_{i_{(A\otimes Y)}}\\
 & B \otimes X \underset{A \otimes X}{\amalg} A\otimes Y \ar@{>-->}[dr]^{f \square g} & \\
 *+<1pc>{B \otimes X} \ar[ur]^{i_{(B \otimes X)}} \ar[rr]_{\id_B \otimes f} && B \otimes Y\,.
}
$$
Using the two-out-of-three property for $S$-local weak equivalences,
we conclude that $f \square g$ is an $S$-local equivalence.
\end{proof}

\begin{theorem}\label{thm:monloc}
Let $S$ be a set of morphisms between cofibrant
objects in $\spref\cC$. Assume that the following condition holds\,:
\begin{itemize}
\item[{\bf (C)}] given an object $\alpha$ in $\cC$ and a map
$G \to H$ in $S$, the morphism
${\alpha} \otimes G \to {\alpha}\otimes  H$
is an $S$-local equivalence.
\end{itemize}
Then the left Bousfield localization $\Loc_S\, \spref\cC$ of $\spref\cC$
with respect to the set $S$, is a symmetric monoidal model category.
\end{theorem}

\begin{proof}
We shall apply Proposition~\ref{prop:loc-mon}.
It is sufficient to prove condition (iii) of \emph{loc. cit.}
In other words, we need to prove that for any
object $F$ of $\spref\cC$ and any map $G\to H$ in $S$, the map
$$F\otimes^\bbL G\too F\otimes^\bbL H$$
in $\Ho(\spref\cC)$ is sent to an isomorphism in
$\Ho(L_S\, \spref\cC)$.
Thanks to Condition {\bf (C)} this is the case when $F$ is representable.
As the functors $(-)\otimes^\bbL F$ commute with
homotopy colimits, the general case follows from
the fact that the functor $\Ho(\spref\cC)\to\Ho(\Loc_S\, \spref\cC)$
commutes with homotopy colimits, and that any simplicial presheaf is a homotopy
colimit of representable presheaves
(see for instance \cite[Proposition~2.9]{Dugger} or \cite[Proposition 3.4.34]{pcmth}).
\end{proof}

\begin{corollary}\label{cor:monloc}
Assume that $S$ is a set of maps in $\cC$ which is closed under
tensor product in $\cC$ (up to isomorphism).
Then, by considering $S$ as a set of maps in $\spref\cC$
via the Yoneda embedding, the left Bousfield localization $\Loc_S\, \spref\cC$
is a symmetric monoidal model category.
\end{corollary}

\begin{proof}
Condition {\bf (C)} of the Theorem~\ref{thm:monloc} is trivially satisfied.
\end{proof}
\section{Monoidal stabilization}\label{chap:stab}
In this section we relate the general theory of spectra with the general theory of symmetric spectra; see Theorem~\ref{thm:main-Hovey}. This will be used in the construction of an explicit symmetric monoidal structure
on the category of non-commutative motives; see Section~\ref{chap:monoidal}.

Let $\cM$ be a cellular (or combinatorial) pointed
simplicial left proper model category.
There is then a natural action of the category of pointed
simplicial sets 
\begin{eqnarray*}
\cM\times \sSet_\bullet\too \cM \ , && (X,K)\longmapsto X\otimes K
\end{eqnarray*}
(in the literature, $X\otimes K$ is usually denoted by
$X\wedge K$).
We denote by $\Spt^{\bbN}(\cM)$ the stable model category of {\em $S^1$-spectra on $\cM$},
where $S^1=\Delta[1]/\partial\Delta[1]$ is the simplicial circle,
seen as endofunctor $X\mapsto X\otimes S^1$ of $\cM$;
see \cite[\S 1]{HoveyS}.

Assume that $\cM$ is a symmetric monoidal model category
with cofibrant unit object ${\bf 1}$.
We write $\Sp(\cM)$ for the stable symmetric monoidal model category
of {\em symmetric $S^1\otimes {\bf 1}$-spectra on $\cM$}; see \cite[\S 7]{HoveyS}. In this situation we have a symmetric monoidal left Quillen functor
\begin{equation*}
\Sigma^\infty:\cM\too \Sp(\cM)\, .
\end{equation*}

\begin{theorem}\label{thm:main-Hovey}
Under the above assumptions, the model categories
$\Spt^{\bbN}(\cM)$ and $\Sp(\cM)$
are canonically Quillen equivalent.
\end{theorem}

\begin{proof}
By applying \cite[Theorem~10.3 and Corollary~10.4]{HoveyS}, it is sufficient to prove that
$S^1\otimes {\bf 1}$ is symmetric in $\Ho(\cM)$, \ie that the permutation $(1,2,3)$
acts trivially on $(S^1\otimes {\bf 1})^{\otimes 3}$.
Using \cite[Corollaire~6.8]{propuni}, we see that 
$(S^1\otimes {\bf 1})^{\otimes 3}\simeq S^3\otimes {\bf 1}$
in $\Ho(\cM)$, and that it is sufficient to check this condition in the case where $\cM$ is
the model category of pointed simplicial sets.
Finally, in this particular case, the result is well known; see for instance \cite[Lemma~6.6.2]{Hovey}.
\end{proof}
\section{Symmetric monoidal structure}\label{chap:monoidal}
In this section we motivate, state and prove our main result:
the localizing motivator carries a canonical symmetric monoidal structure;
see Theorem~\ref{thm:main-mon}.
\begin{definition}\label{def:ses}
A sequence of triangulated categories
$$ \cR \stackrel{I}{\too} \cS \stackrel{P}{\too} \cT$$
is called {\em exact} if the composition is zero, the functor $I$ is
fully-faithful and the induced functor from the Verdier quotient $\cS/\cR$ to $\cT$
is {\em cofinal}, \ie it is fully-faithful and every object in $\cT$ is a direct summand
of an object of $\cS/\cR$; see \cite[\S2]{Neeman} for details. A sequence of dg categories
$$ \cA \stackrel{F}{\too} \cB \stackrel{G}{\too} \cC$$
is called {\em exact} if the induced sequence of triangulated categories (see \S\ref{sub:modules})
$$ \cD(\cA) \stackrel{\bbL F_!}{\too} \cD(\cB) \stackrel{\bbL G_!}{\too} \cD(\cC)$$
is exact.
\end{definition}
Recall from \cite[\S 10]{Duke} the construction of the {\em universal localizing invariant}.

\begin{theorem}{(\cite[Theorem~10.5]{Duke})}\label{thm:Uloc}
There exists a morphism of derivators
$$ \Uloc: \HO(\dgcat) \too \Mloc\,,$$
with values in a strong triangulated derivator (see \S\ref{sub:properties}),
which has the following properties\,:
\begin{itemize}
\item[$\flt$)] $\Uloc$ preserves filtered homotopy colimits;
\item[$\p$)] $\Uloc$ preserves the terminal object;
\item[$\loc$)] $\Uloc$ satisfies {\em localization}, \ie sends exact sequence of dg categories
$$ \cA \too \cB \too \cC $$
to distinguished triangles in $\Mloc(e)$
$$\Uloc(\cA) \too \Uloc(\cB) \too \Uloc(\cC) \too \Uloc(\cA)[1]\,.$$
\end{itemize}
Moreover, $\Uloc$ is universal with respect to these properties, $\ie$ given any strong
triangulated derivator $\bbD$, we have an induced equivalence of categories
$$ (\Uloc)^{\ast}: \HomC(\Mloc, \bbD) \stackrel{\sim}{\too} \HomL(\HO(\dgcat), \bbD)\,,$$
where the left-hand side stands for the category of homotopy colimit
preserving morphisms of derivators, {while the
right-hand side stands for the full subcategory of $\uHom(\HO(\dgcat), \bbD)$
spanned by the morphisms of derivators which verify the three conditions above.}
\end{theorem}

\begin{definition}\label{not:Linv}
The objects of the category $\HomL(\HO(\dgcat), \bbD)$ are called {\em localizing invariants}
and $\Uloc$ is called the {\em universal localizing invariant}. Because of its universal property,
which is a reminiscence of motives, $\Mloc$ is called the {\em localizing motivator}.
Its base category $\Mloc(e)$ is a triangulated category and
is morally what we would like to consider as the category
of {\em non-commutative motives}.
\end{definition}

\begin{example}\label{ex:examples-loc}
{Examples of localizing invariants include Hochschild homology and cyclic homology (see \cite[Theorem~10.7]{Duke})}, non-connective algebraic
$K$-theory (see \cite[Theorem~10.9]{Duke}), and even topological Hochschild homology and
topological cyclic homology (see \cite[Theorem~6.1]{Blum-Mandell} and \cite[\S8]{AGT}).
\end{example}
In this section we introduce a new ingredient in Theorem~\ref{thm:Uloc}:
symmetric monoidal structures. As shown in Theorem~\ref{thm:Mdgcat} the derivator
$\HO(\dgcat)$ carries a symmetric monoidal structure. It is therefore natural to consider
localizing invariants which are symmetric monoidal; {see Examples \ref{ex:HH}-\ref{ex:Periodiccomp}}.
Our main result is the following.
\begin{theorem}\label{thm:main-mon}
The localizing motivator $\Mloc$ carries a canonical symmetric mo\-no\-idal struc\-ture
$-\otimes^{\bbL}-$, making the universal localizing invariant $\Uloc$
symmetric monoidal (see \S\ref{sub:M1}). Moreover, this tensor product
preserves homotopy colimits in each variable, and is characterized by the following universal property:
given any strong triangulated derivator $\bbD$ (see \S\ref{sub:properties}), endowed with
a symmetric monoidal structure which preserves homotopy colimits in each
variable, we have an induced equivalence of categories
\begin{equation*}
(\Uloc)^{\ast}: \HomC^{\otimes}(\Mloc, \bbD) \stackrel{\sim}{\too} \HomL^{\otimes}(\HO(\dgcat), \bbD)\, ,
\end{equation*}
where the left-hand side stands for the category of symmetric monoidal homotopy
colimit preserving morphisms of derivators, while
the right-hand side stands for the category of symmetric monoidal
morphisms of derivators which belong to the category $\HomL(\HO(\dgcat), \bbD)$.
Furthermore, $\Mloc$ admits an explicit symmetric monoidal Quillen model.
\end{theorem}

\begin{definition}\label{not:Linv-mon}
The objects of the category $\HomL^{\otimes}(\HO(\dgcat), \bbD)$ are
called {\em symmetric monoidal localizing invariants}.
\end{definition}

\begin{corollary}\label{satcompact}
Any dualizable object of $\Mloc(e)$ is compact; see Definition~\ref{def:dualizable}
and \S\ref{sub:triangulated}.
In particular, given any saturated dg category $\cA$ (see Definition~\ref{def:saturated}),
the object $\Uloc(\cA)$ is compact.
\end{corollary}

\begin{proof}
Let $M$ be a dualizable object of $\Mloc(e)$.
We need to prove that the functor $\Hom(M,-)$
commutes with arbitrary sums. Since $M$ is dualizable, this functor
is isomorphic to $\Hom(\Uloc(\underline{k}),M^\vee\otimes^\bbL -)$.
The unit object $\Uloc(\underline{k})$ is known to be compact (see \cite[Theorem 7.16]{CT}),
and so the first assertion is proven.
The second assertion follows from the fact that $\Uloc$ is symmetric monoidal,
and that $\Uloc(\cA)$ is dualizable for any saturated dg category $\cA$
(see Remark \ref{rk:adj-dualizable}~(iv)
and Theorem \ref{thm:Toen-dualizable}).
\end{proof}

\begin{remark}
Although we do not know if the triangulated category $\Mloc(e)$ is
compactly generated, Corollary~\ref{satcompact} implies that
the localizing triangulated subcategory of $\Mloc(e)$ generated by dualizable objects
is compactly generated.
\end{remark}
{Before proving Theorem~\ref{thm:main-mon}, let us give some examples of symmetric monoidal localizing invariants.

\begin{example}[Hochschild homology]\label{ex:HH}
Let $\cA$ be a small $k$-flat dg category; see Definition~\ref{def:k-flatdg}.
We can associate to $\cA$ a simplicial object in $\cC(k)$, \ie a contravariant functor from $\Delta$ to
$\cC(k)$ : its $n$-th term is given by
$$ \underset{(x_0, \ldots, x_n)}{\bigoplus} \cA(x_n, x_0) \otimes \cA(x_{n-1}, x_n) \otimes
\cdots \otimes \cA(x_0, x_1)\,,$$
where $(x_0, \ldots, x_n)$ is an ordered sequence of objects in $\cA$. The face maps are given by 
\begin{equation*}
d_i(f_n, \ldots, f_i, f_{i-1}, \ldots, f_0) = \left\{ \begin{array}{lcr}
(f_n, \ldots, f_i \circ f_{i-1}, \ldots , f_0) & \text{if} & i>0 \\
(-1)^{(n + \sigma)}(f_0\circ f_n, \ldots, f_1) & \text{if} & i=0 \\
\end{array} \right. 
\end{equation*}
where $\sigma=(\mathrm{deg}f_0)(\mathrm{deg}f_1 + \ldots + \mathrm{deg}f_{n-1})$, and the
degenerancies maps are given by 
$$ s_j(f_n, \ldots, f_j, f_{j-1}, \ldots, f_0)=(f_n, \ldots, f_j, \id_{x_j}, f_{j-1}, \ldots, f_0)\,.$$
Associated to this simplicial object we have a chain complex in $\cC(k)$ (by the Dold-Kan equivalence),
and so a bicomplex of $k$-modules. The {\em Hochschild complex} $\mathit{HH}(\cA)$ of $\cA$ is
the sum-total complex associated to this bicomplex. If $\cA$ is an arbitrary dg category, its
Hochschild complex is obtained by first taking a $k$-flat (\eg~cofibrant) resolution of $\cA$;
see Definition~\ref{def:k-flatdg}. This construction furnishes us a functor
$$ \mathit{HH}: \dgcat \too \cC(k) \,,$$
which by \cite[Theorem~10.7]{Duke} gives rise to a localizing invariant 
\begin{equation}\label{eq:HHder}
\mathit{HH}: \HO(\dgcat) \too \HO(\cC(k))\,.
\end{equation}
Given small dg categories $\cA$ and $\cB$, we have a functorial quasi-isomorphism
$$ sh: \mathit{HH}(\cA) \otimes \mathit{HH}(\cB) \too \mathit{HH}(\cA \otimes \cB)$$
given by the shuffle product map; see \cite[\S4.2.3]{Loday}.
The localizing invariant (\ref{eq:HHder}),
endowed with the shuffle product map, becomes then a symmetric monoidal localizing invariant.
\end{example}
\begin{example}[Mixed complexes]\label{ex:Mixedcomp}
Following Kassel~\cite[\S1]{Kassel} we denote by $\Lambda$ the dg algebra $k[\epsilon]/\epsilon^2$, where $\epsilon$ is of degree $-1$ and $d(\epsilon)=0$. Under this notation, a {\em mixed complex} is a right dg $\Lambda$-module (see Definition~\ref{def:modules}).

Let $\cA$ be a small dg category. The Hochschild complex $\mathit{HH}(\cA)$ of $\cA$ (see Example~\ref{ex:HH}), endowed with the cyclic operator
$$ t_n(f_{n-1}, \ldots, f_0) = (-1)^{n+\sigma}(f_0,f_{n-1}, f_{n-2}, \ldots, f_1)\,,$$
gives rise to a mixed complex $\mathit{C}(\cA)$; see \cite[\S1.3]{Exact}. The assignment $\cA \mapsto \mathit{C}(\cA)$ yields a localizing invariant
\begin{equation}\label{eq:mixed}
\mathit{C}: \HO(\dgcat) \too \HO(\cC(\Lambda))
\end{equation}
with values in the derivator associated to right dg $\Lambda$-modules; see \cite[Theorem~10.7]{Duke}. Recall from \cite[\S1]{Kassel} that the category $\cC(\Lambda)$ carries a natural symmetric monoidal structure whose unit object is $k$: given two mixed complexes there is a canonical mixed complex structure on the tensor product of the underlying complexes. Moreover, this symmetric monoidal structure is compatible with the projective model structure (see \S\ref{sub:modules}). Thanks to \cite[Theorem~2.4]{Kassel} the localizing invariant (\ref{eq:mixed}) becomes then a symmetric monoidal localizing invariant. 

This example will be used in the construction of a canonical Chern character map from non-connective algebraic $K$-theory to negative cyclic homology; see Example~\ref{ex:negative}. 
\end{example}
\begin{example}[Periodic complexes]\label{ex:Periodiccomp}
In this example we assume that our base ring $k$ is a field. Let $k[u]$ be the cocommutative
Hopf algebra of polynomials in one variable $u$ of degree $2$; see \cite[\S1]{Kassel}.
Consider the symmetric monoidal model category $k[u]\text{-}\mathrm{Comod}$ of
$k[u]$-comodules; see \cite[Theorem~2.5.17]{Hovey}. The monoidal structure is given
by the cotensor product $-\square_{k[u]}-$ of comodules, with unit $k[u]$.

Given a mixed complex $M$ (see Example~\ref{ex:Mixedcomp}) we denote by
$\mathit{P}(M)$ the $k[u]$-comodule, whose underlying complex is $M \otimes^{\bbL}_{\Lambda}k$,
obtained by iteration of the map
$$ (M \otimes^{\bbL}_{\Lambda}k)[-2] \stackrel{S}{\too} M\otimes^{\bbL}_{\Lambda}k\,,$$
see \cite[Proposition~1.4]{Kassel}. Using \cite[Theorem~1.7]{Kassel}
and \cite[Proposition 9.2]{EilMoo}, we conclude that we have a symmetric
monoidal morphism of derivators
$$\mathit{P}:\HO(\cC(\Lambda)) \too \HO(k[u]\text{-}\mathrm{Comod})\,.$$
By composing $\mathit{P}$ with the localizing invariant (\ref{eq:mixed}) we obtain then a
symmetric monoidal localizing invariant
$$ (\mathit{P}\circ \mathit{C}): \HO(\dgcat) \too \HO(k[u]\text{-}\mathrm{Comod})\,.$$
This example will be used in the construction of a canonical Chern character map from non-connective
algebraic $K$-theory to periodic cyclic homology; see Example~\ref{ex:periodic}. 
\end{example}
}
\subsection{Proof of Theorem~\ref{thm:main-mon}}\label{sub:proof-main}
We will use freely the theory of derivators which is recalled and developed in the appendix.
Recall from \cite[\S10]{Duke} that $\Uloc$ is obtained by the following composition\,:
$$ \HO(\dgcat) \stackrel{\bbR\underline{h}}{\too} \Loc_{\Sigma}\mathsf{Hot}_{\dgcat_{\f}}
\stackrel{\Phi}{\too} 
\Loc_{\Sigma, P}\mathsf{Hot}_{\dgcat_{\f}} \stackrel{\stab}{\too}
\St(\Loc_{\Sigma, P}\mathsf{Hot}_{\dgcat_{\f}}) \stackrel{\gamma}{\too} \Mloc\,.$$
The morphism $\stab$ corresponds to a stabilization procedure (see \S\ref{sub:Heller})
and the morphism $\gamma$ to a left Bousfield localization procedure (see \S\ref{sub:leftBousfield}).
Since these procedures commute (see Proposition~\ref{commut}), we can also obtain $\Uloc$ by the
following composition
\begin{equation}\label{eq:comp-new}
\HO(\dgcat) \stackrel{\bbR\underline{h}}{\too}
\Loc_{\Sigma}\mathsf{Hot}_{\dgcat_{\f}} \stackrel{\Phi}{\too} 
\Loc_{\Sigma, P}\mathsf{Hot}_{\dgcat_{\f}} \stackrel{\gamma}
{\too} \Mlocu \stackrel{\stab}{\too} \Mloc\,,
\end{equation}
where $\Mlocu$ is the {\em unstable} analogue of the localizing motivator.

The proof of Theorem~\ref{thm:main-mon} will consist on the concatenation of
Propositions~\ref{prop:FltMon}-\ref{prop:stabMon} followed by
Remark~\ref{rk:explanations}. In each one of these propositions we construct
an explicit symmetric monoidal Quillen model for the corresponding (intermediate)
derivator of the composition (\ref{eq:comp-new}).

We start by fixing on $\dgcat$ a fibrant resolution functor $R$,
a cofibrant resolution functor $Q$,
a left framing $\Gamma_*$ (\ie a well-behaved cosimplicial resolution functor;
see \cite[Definition~5.2.7 and Theorem~5.2.8]{Hovey}),
as well as a small full subcategory $\dgcat_{\f}$ of $\dgcat$, satisfying the following properties\,:
\begin{itemize}
\item[(a)] any finite dg cell (see Definition~\ref{def:dg-cell}) is in $\dgcat_{\f}$;
\item[(b)] any object in $\dgcat_{\f}$ is homotopically finitely presentated (see Definition~\ref{def:HFP});
\item[(c)] given any object $\cA$ in $\dgcat_{\f}$, $Q(R(\cA))$ and $Q(\cA)$
belong to $\dgcat_{\f}$;
\item[(d)] for any cofibrant object $\cA$ of $\dgcat_{\f}$, if
$\Gamma_*(\cA)$ denotes the given cosimplicial frame of $\cA$, then
$\Gamma_n(\cA)$ belongs to $\dgcat_{\f}$ for all $n\geq 0$.
\end{itemize}
We let $\Sigma$ be the set of derived Morita
equivalences in $\dgcat_{\f}$.
The derivator $\Loc_\Sigma\mathsf{Hot}_{\dgcat_{\f}}$ is simply the
derivator $\HO(\Loc_\Sigma \, \spref{\dgcat_{\f}})$ associated to the
left Bousfield localization of the projective model structure on $\spref{\dgcat_{\f}}$
(see \S\ref{sub:modelBK}), with respect to the set $\Sigma$.
Note that, up to Quillen equivalence, this construction
does not depend on the choice of the category
$\dgcat_{\f}$ but only on the Dwyer-Kan
localization of $\dgcat_{\f}$ by $\Sigma$ (see \cite{DwKan}). The above stability properties
imply that the Dwyer-Kan localization of $\dgcat_{\f}$ by $\Sigma$
is simply (equivalent to) the full simplicial subcategory
of the Dwyer-Kan localization of the model category $\dgcat$
spanned by the homotopically finitely presented dg categories.
In order to obtain a symmetric monoidal structure, we have then the freedom to
add the following properties to $\dgcat_{\f}$\,:
\begin{itemize}
\item[(e)] any dg category in $\dgcat_{\f}$ is $k$-flat (see Definition~\ref{def:k-flatdg});
\item[(f)] given any dg categories $\cA$ and $\cB$ in $\dgcat_{\f}$, $\cA\otimes \cB$
belongs to $\dgcat_{\f}$ (this makes sense because of Theorem~\ref{thm:tensorstable}).
\end{itemize}
In the sequel, we assume that a small full subcategory $\dgcat_{\f}$
of $\dgcat$ satisfying all the above properties (a)--(f) has been chosen; for instance, one might
consider the smallest one relatively to $R$, $Q$ and $\Gamma_*$.
The morphism
$$\bbR\underline{h} : \HO(\dgcat)\too\Loc_{\Sigma}\mathsf{Hot}_{\dgcat_{\f}}
=\HO(\Loc_\Sigma \spref\dgcat_{\f})$$
is induced by the functor
$$\underline{h}:\dgcat\too \spref\dgcat_{\f}\,,$$
which associates to any dg category $\cA$ the simplicial
presheaf on $\dgcat_{\f}$\,:
$$\bbR\underline{h}(\cA): \cB\longmapsto \Map(\cB,\cA)=\Hom(\Gamma_*(Q(\cB)),R(\cA))\, .$$

\begin{proposition}\label{prop:FltMon}
The tensor product of dg categories in $\dgcat_{\f}$
extends uniquelly to a closed symmetric monoidal structure
on the category of simplicial presheaves on $\dgcat_{\f}$, making $\Loc_\Sigma\,  \spref\dgcat_{\f}$
into a symmetric monoidal
model category. As a consequence, the derivator
$\Loc_{\Sigma}\mathsf{Hot}_{\dgcat_{\f}}$ carries a symmetric
monoidal structure, making the morphism $\bbR \underline{h}$
symmetric monoidal. Moreover, given any derivator $\bbD$,
the category of filtered homotopy colimit preserving
symmetric monoidal morphisms
from $ \HO(\dgcat)$ to $\bbD$ is canonically equivalent
to the category of homotopy colimit preserving symmetric monoidal
morphisms from $\Loc_{\Sigma}\mathsf{Hot}_{\dgcat_{\f}}$ to $\bbD$.
\end{proposition}

\begin{proof}
As derived Morita equivalences are stable under
derived tensor product, it follows immediately from
Corollary \ref{cor:monloc} that $\Loc_\Sigma\,  \spref\dgcat_{\f}$
is a symmetric monoidal model category.

Let us now show that the morphism $\bbR\underline{h}$ is symmetric monoidal.
Recall from \cite[\S5]{Duke} that the morphism $\bbR\underline{h}$ preserves
filtered homotopy colimits and that we have a commutative diagram
\begin{equation}\label{eq:filtered}
\xymatrix{
\underline{\dgcat_{\f}}[\Sigma^{-1}] \ar[d] \ar[r]^i & \HO(\dgcat) \ar[dl]^{\bbR\underline{h}} \\
\Loc_{\Sigma} \mathsf{Hot}_{\dgcat_{\f}} & ,
}\end{equation}
where $\underline{\dgcat_{\f}}$ stands for the prederivator
represented by ${\dgcat_{\f}}$, and $\underline{\dgcat_{\f}}[\Sigma^{-1}]$
for its formal localization by $\Sigma$.
By construction, the left vertical morphism in the above diagram (\ref{eq:filtered}) is symmetric monoidal and the symmetric monoidal
structure on $\Loc_{\Sigma} \mathsf{Hot}_{\dgcat_{\f}}$ preserves homotopy colimits
in each variable. Moreover, thanks to Lemma~\ref{prop:Fltdgcat}, the symmetric monoidal
structure on $\HO(\dgcat)$ preserves filtered homotopy colimits.

Now, recall the universal property of $\HO(\dgcat)$\,: it is the free completion of the prederivator
$\underline{\dgcat_{\f}}[\Sigma^{-1}]$ by filtered
homotopy colimits. In other words, given any derivator $\bbD$, the category
of filtered homotopy colimit preserving morphisms
from $\HO(\dgcat)$ to $\bbD$ is canonically equivalent
to the category of morphisms from $\underline{\dgcat_{\f}}[\Sigma^{-1}]$ to $\bbD$; see \cite[\S5]{Duke}.
Replacing $\bbD$ by the derivator of filtered homotopy colimit
preserving morphisms from $\HO(\dgcat)$ to $\bbD$, we deduce that
the category of morphisms from $\HO(\dgcat)\times \HO(\dgcat)$ to $\bbD$
which preserve filtered homotopy colimits in each variable
is equivalent to the category of morphisms from
$\underline{\dgcat_{\f}}[\Sigma^{-1}]\times \underline{\dgcat_{\f}}[\Sigma^{-1}]$
to $\bbD$. By induction, we prove similarly that, for any $n\geq 0$,
the category of morphisms from $\HO(\dgcat)^n$ to $\bbD$
which preserve filtered homotopy colimits in each variable
is equivalent to the category of morphisms
from $\underline{\dgcat_{\f}}[\Sigma^{-1}]^n$ to $\bbD$.
As the morphism $i$ in the above diagram \eqref{eq:filtered}
is symmetric monoidal, this implies that the morphism $\bbR\underline{h}$
is symmetric monoidal as well.
Similarly, we see that, given any derivator $\bbD$, the
category of symmetric monoidal morphisms
from $\underline{\dgcat_{\f}}[\Sigma^{-1}]$ to $\bbD$
is equivalent to the category of filtered homotopy colimit
preserving symmetric monoidal morphisms from
$\HO(\dgcat)$ to $\bbD$.
The last assertion of this proposition thus follows
from Theorem \ref{monoidalderKanext} and
Proposition \ref{derivatorleftlocmonoidal}.
\end{proof}

Let $h:\dgcat_{\f}\too\spref\dgcat_{\f}$ be the Yoneda embedding.
We denote by $P:\varnothing\too h(\varnothing)$ the canonical map.
Then, the derivator $\Loc_{\Sigma,P} \mathsf{Hot}_{\dgcat_{\f}}$
is simply the left Bousfield localization of $\Loc_{\Sigma} \mathsf{Hot}_{\dgcat_{\f}}$
by $P$. Thus, it can be described as
$$\Loc_{\Sigma,P} \mathsf{Hot}_{\dgcat_{\f}}=\HO(\Loc_{\Sigma,P} \spref\dgcat_{\f})\, ,$$
where $\Loc_{\Sigma,P} \spref\dgcat_{\f}$ is the left Bousfield localization
of the model category $\Loc_\Sigma \spref\dgcat_{\f}$ by the map $P$.

\begin{proposition}\label{prop:FltMonP}
The model category $\Loc_{\Sigma,P}\,  \spref\dgcat_{\f}$
is symmetric monoidal, and the localization
functor
$$\Loc_{\Sigma} \, \spref\dgcat_{\f}\too \Loc_{\Sigma,P}\, \spref\dgcat_{\f}$$
is a symmetric monoidal left Quillen functor.
In particular, the derivator
$\Loc_{\Sigma,P} \mathsf{Hot}_{\dgcat_{\f}}$ is symmetric monoidal, and
the localization morphism
$$\Phi: \Loc_{\Sigma} \mathsf{Hot}_{\dgcat_{\f}}\too
\Loc_{\Sigma,P} \mathsf{Hot}_{\dgcat_{\f}}$$
is symmetric monoidal.
\end{proposition}
\begin{proof}
For any dg category $\cA$, we have $\cA\otimes\varnothing\simeq\varnothing$.
We deduce easily from this formula that condition \textbf{(C)}
of Theorem \ref{thm:monloc} (with $\cM=\Loc_{\Sigma}\,  \spref\dgcat_{\f}$) is satisfied, and so the proof is finished.
\end{proof}

Let  $\spref\dgcat_{\f,\bullet}$ be the model category of
pointed simplicial presheaves on $\spref\dgcat_{\f}$.
By virtue of Proposition \ref{prop:Mprojectivepointed}, this
is a symmetric monoidal model category, and the functor
$$\spref\dgcat_{\f}\too\spref\dgcat_{\f,\bullet} \ , \quad F\longmapsto F_+$$
is a symmetric monoidal left Quillen functor.
We define the pointed model category $\Loc_{\Sigma,P}\, \spref\dgcat_{\f,\bullet}$
as the left Bousfield localization of $\spref\dgcat_{\f,\bullet}$ with respect to the set of maps $\Sigma_+\cup\{P_+\}$.

\begin{proposition}\label{prop:FltMonP2}
The model category $\Loc_{\Sigma,P}\, \spref\dgcat_{\f,\bullet}$
is symmetric monoidal, and the symmetric monoidal left Quillen functor
$$\Loc_{\Sigma,P} \, \spref\dgcat_{\f}\too \Loc_{\Sigma,P}\, \spref\dgcat_{\f,\bullet}$$
is a Quillen equivalence. In particular, we have a canonical
equivalence of symmetric monoidal derivators
$$\Loc_{\Sigma,P} \mathsf{Hot}_{\dgcat_{\f}}\simeq \HO(\Loc_{\Sigma,P}\, \spref\dgcat_{\f,\bullet})\, .$$
\end{proposition}

\begin{proof}
The first assertion is a direct application of Theorem \ref{thm:monloc},
while the second one follows from \cite[Remark 8.2]{Duke}.
\end{proof}

Note that the initial and terminal dg categories are Morita
equivalent. This implies that the dg category $0$
is sent to the point (up to weak equivalence) in $\Loc_{\Sigma,P}\, \spref\dgcat_{\f,\bullet}$.
Let $\cE$ be the set of morphisms of $\Loc_{\Sigma,P}\, \spref\dgcat_{\f,\bullet}$
of shape
$$\mathsf{cone}[\bbR\underline{h}(\cA)\too \bbR\underline{h}(\cB)]\too
\bbR\underline{h}(\cC)\, ,$$
associated to each exact sequence of dg categories
$$\cA\too\cB \too\cC\, ,$$
with $\cB$ in $\dgcat_{\f}$ (where $\mathsf{cone}$ means homotopy cofiber).
We define $\cMlocu$ as the left Bousfield localization of
$\Loc_{\Sigma,P}\spref\dgcat_{\f,\bullet}$ by $\cE$.
The derivator $\Mlocu$ is defined as
$$\Mlocu=\HO(\cMlocu)\, .$$

\begin{proposition}\label{prop:LocMon}
The model category $\cMlocu$ is symmetric monoidal, in such a way that
the left Quillen functor
$$\Loc_{\Sigma,P}\, \spref\dgcat_{\f,\bullet}\too
\cMlocu$$
is symmetric monoidal. Under the identification of Proposition \ref{prop:FltMonP},
the induced morphism of derivators
$$\gamma :\Loc_{\Sigma,P} \mathsf{Hot}_{\dgcat_{\f}} \too \Mlocu$$
is symmetric monoidal.
\end{proposition}

\begin{proof}
As tensoring by a $k$-flat dg category preserves
exact sequences of dg categories (see \cite[Proposition~1.6.3]{Drinfeld}),
and as $\bbR\underline{h}$ is symmetric monoidal (see Proposition \ref{prop:FltMon}), the proof follows from Theorem~\ref{thm:monloc}.
\end{proof}

Finally, since by construction the model category
$\cMlocu$ is symmetric monoidal and simplicially enriched,
we can consider its stabilization $\cMloc$, \ie the
stable model category of symmetric spectra in $\cMlocu$ (see \S\ref{chap:stab})\,:
$$\cMloc=\Sp(\cMlocu)\, .$$
The derivator $\Mloc$ is defined as
$$\Mloc=\HO(\cMloc)\, .$$

\begin{proposition}\label{prop:stabMon}
The model category $\cMloc$ is symmetric monoidal, and the left Quillen
functor
$$\Sigma^\infty:\cMlocu\too\cMloc$$
is symmetric monoidal.
The induced morphism of derivators
$$\stab=\bbL\Sigma^\infty:\Mlocu\too\Mloc$$
is symmetric monoidal.
\end{proposition}

\begin{proof}
This is true by construction (see Proposition~\ref{prop:Mderivator}).
\end{proof}

\begin{remark}
In the construction of $\Mloc$ given in \cite{Duke}, the definition
of $\Mloc$ was $\HO(\Spt^{\bbN}(\cMlocu))$, \ie it used
non-symmetric spectra. However, thanks to
Theorem \ref{thm:main-Hovey}, both definitions agree up to a canonical equivalence
of derivators.
\end{remark}

\begin{remark}\label{rk:explanations}
The concatenation of Propositions~\ref{prop:FltMon}-\ref{prop:stabMon}
show us that the localizing motivator $\Mloc$ carries a symmetric monoidal structure
$-\otimes^{\bbL}-$, making the universal localizing invariant $\Uloc$ symmetric
monoidal. By Proposition~\ref{prop:Mderivator} the
associated symmetric monoidal structure preserves homotopy colimits.
Therefore, in order to conclude the proof of Theorem~\ref{thm:main-mon},
it remains to show the universal property. Let $\bbD$ be a strong triangulated
derivator endowed with a symmetric monoidal structure which preserves homotopy
colimits in each variable.
Thanks to Theorem~\ref{thm:Uloc}, we have an induced equivalence of categories
\begin{equation*}
(\Uloc)^{\ast}: \HomC(\Mloc, \bbD) \stackrel{\sim}{\too} \HomL(\HO(\dgcat), \bbD)\,.
\end{equation*}
This implies that the induced functor 
\begin{equation}\label{eq:equivalence1}
(\Uloc)^{\ast}: \HomC^{\otimes}(\Mloc, \bbD) \too \HomL^{\otimes}(\HO(\dgcat), \bbD)\,.
\end{equation}
is faithful. More precisely, by Proposition \ref{prop:FltMon},
the category of filtered homotopy
colimit preserving symmetric monoidal morphisms
from $\HO(\dgcat)$ to $\bbD$ is equivalent to the category
of homotopy colimit preserving symmetric monoidal morphisms
from $\Loc_{\Sigma}\spref\dgcat_{\f}$ to $\bbD$.
Using the universal properties of Bousfield localization
and stabilization in the setting of derivators (see Theorem~\ref{thm:Quillenloc} and
Corollary~\ref{stabcombmodcat}), we can apply Proposition \ref{derivatorleftlocmonoidal},
and Theorem \ref{stableDay} to conclude, by
construction of $\Mloc$, that \eqref{eq:equivalence1}
is an equivalence of categories. This ends
the proof of Theorem \ref{thm:main-mon}.
\end{remark}

\section{Applications}\label{sub:applications}
In this section we describe several applications of Theorem~\ref{thm:main-mon}.
\subsection{Non-connective $K$-theory}
Recall from Example~\ref{ex:examples-loc} that non-connective algebraic $K$-theory is an
example of a localizing invariant of dg categories. In \cite{CT} the authors proved that this
localizing invariant becomes co-representable in the category $\Mloc(e)$ of non-commutative
motives (see Definition~\ref{not:Linv}).
\begin{theorem}{(\cite[Theorem~7.16]{CT})}\label{thm:co-repres}
For every small dg category $\cA$, we have a natural isomorphism in the stable homotopy
category of spectra
\begin{equation*}
\bbR\Hom(\,\Uloc(\underline{k}),\, \Uloc(\cA)\,) \simeq \bbK(\cA)\,.
\end{equation*}
Here, $\underline{k}$ denotes the dg category with one object $\ast$ such that
$\underline{k}(\ast,\ast)=k$ in degree zero (see \S\ref{sub:dg-cells}(i)), and $\bbK(\cA)$
the non-connective algebraic $K$-theory spectrum of $\cA$. In particular, we obtain isomorphisms
of abelian groups
\begin{eqnarray*}
\Hom(\,\Uloc(\underline{k})[n],\, \Uloc(\cA)\,)
\simeq \bbK_n(\cA)&& n \in \bbZ\,.
\end{eqnarray*}
\end{theorem}
A fundamental problem of the theory of non-commutatives motives is the computation of
the (spectra of) morphisms between two object in the localizing motivator. Using
Theorem~\ref{thm:main-mon} we give a partial solution to this fundamental problem.
\begin{theorem}\label{thm:co-repres-ext}
Let $\cB$ be a saturated dg category; see Definition~\ref{def:saturated}. For every
small dg category $\cA$, we have a natural isomorphism in the stable homotopy category of spectra
\begin{equation*}
\bbR\Hom(\,\Uloc(\cB),\, \Uloc(\cA)\,) \simeq \bbK(\rep(\cB, \cA))\,.
\end{equation*}
Here $\rep(-,-)$ denotes the internal Hom-functor in $\Hmo$; see Theorem~\ref{thm:Toen}.
In particular, we obtain isomorphisms of abelian groups
\begin{eqnarray*}
\Hom(\,\Uloc(\cB)[n],\, \Uloc(\cA)\,)
\simeq \bbK_n(\rep(\cB, \cA))&& n \in \bbZ\,.
\end{eqnarray*}
\end{theorem}
\begin{proof}
The proof is a consequence of the following weak equivalences\,:
\begin{eqnarray}
\bbR\Hom(\,\Uloc(\cB), \, \Uloc(\cA)\,) & \simeq & \bbR\Hom(\, \Uloc(\underline{k})\otimes^{\bbL} \Uloc(\cB), \,\Uloc(\cA)) \label{eq:I}\\
& \simeq & \bbR\Hom(\, \Uloc(\underline{k}), \,   \Uloc(\cB)^{\vee} \otimes^{\bbL} \Uloc(\cA) ) \label{eq:II} \\
& \simeq & \bbR\Hom(\, \Uloc(\underline{k}), \,   \Uloc(\cB^\op  \otimes^{\bbL} \cA)) \label{eq:IV} \\
& \simeq & \bbR\Hom(\, \Uloc(\underline{k}), \, \Uloc(\rep(\cB,\cA))) \label{eq:V} \\
& \simeq & \bbK(\rep(\cB, \cA)) \label{eq:VI} \,. \\ \nonumber 
\end{eqnarray}
Equivalence (\ref{eq:I}) follows from the fact that $\Uloc(\underline{k})$ is the unit object in
$\Mloc(e)$; see Remark~\ref{rk:monoid-str} and Theorems~\ref{thm:Mdgcat} and \ref{thm:main-mon}.
Since $\cB$ is a saturated dg category, Theorem~\ref{thm:Toen-dualizable} implies that $\cB$ is a
dualizable object in $\Hmo$. Therefore, Equivalence (\ref{eq:II}) follows from the fact that
$\Uloc(\cB)$ is a dualizable object in $\Mloc(e)$ (see Remark~\ref{rk:adj-dualizable}(iv) and
Theorem~\ref{thm:main-mon}) and from the adjunction (\ref{eq:adj-dualizable}) of
Remark~\ref{rk:adj-dualizable}(ii) (see \S\ref{sub:spectral-enr} for its spectral enrichment).
Equivalence (\ref{eq:IV}) follows from Remark~\ref{rk:adj-dualizable}(iv), from
Theorem~\ref{thm:Toen-dualizable}, and from the fact that the universal localizing invariant is
symmetric monoidal. Equivalence (\ref{eq:V}) follows from Lemma~\ref{lem:tensor-rep}.
Finally, Equivalence (\ref{eq:VI}) follows from Theorem~\ref{thm:co-repres}.
\end{proof}

\begin{proposition}\label{prop:co-repres-schemes}
Let $X$ and $Y$ be smooth and proper $k$-schemes. Then, we have a natural isomorphism in the
stable homotopy category of spectra
$$\bbR\Hom(\, \Uloc(\perf(X)), \Uloc(\perf(Y)) \,) \simeq \bbK(X \times Y)\,.$$
Here, $\bbK(X \times Y)$ denotes the non-connective algebraic $K$-theory spectrum of $X \times Y$ (see \cite[\S8]{Marco}), and $\perf(-)$ the dg category constructed in
Example~\ref{ex:Toen}(i). In particular, we obtain isomorphisms of abelian groups
\begin{eqnarray*}
\Hom(\,\Uloc(\perf(X))[n],\, \Uloc(\perf(Y))\,)
\simeq \bbK_n(X \times Y)&& n \in \bbZ\,.
\end{eqnarray*}
\end{proposition}
\begin{proof}
Since $X$ and $Y$ are smooth and proper $k$-schemes, \cite[Lemma~3.27]{Moduli} implies
that $\perf(X)$ and $\perf(Y)$ are saturated dg categories. Therefore, by
Theorem~\ref{thm:co-repres-ext} we have a natural isomorphism in the stable homotopy category of spectra
$$\bbR\Hom(\, \Uloc(\perf(X)), \Uloc(\perf(Y)) \,) \simeq \bbK(\rep(\perf(X),\perf(Y)))\,.$$
Moreover, by \cite[Theorem~8.9]{Toen} we have a natural isomorphism
$$ \perf(X\times Y) \simeq \rep(\perf(X), \perf(Y))$$
in $\Hmo$. Finally, thanks to \cite[\S8\,\,Theorem~5]{Marco} we have a natural isomorphism
$$ \bbK(\perf(X \times Y)) \simeq \bbK(X \times Y)$$
and so the proof is finished.
\end{proof}
{
\begin{remark}
Let $Z$ be a noetherian regular scheme. Thanks to \cite[Exp.~I, Cor.~5.9 and Exp.~II, Cor.~2.2.2.1]{SGA6}
we have a derived Morita equivalence
$$ \perf(Z) \stackrel{\sim}{\too} \cD^b_{\dg}(\mathit{Coh}(Z))\,,$$
where the left-hand side is the saturated dg category of Example~\ref{ex:Toen}(i) and the right-hand side is the bounded derived (dg) category of coherent sheaves on $Z$.
Since $\mathit{Coh}(Z)$ is a noetherian abelian category (see \cite[\S10.1]{Marco}),
we conclude by \cite[\S10.1\,\,Theorem~7]{Marco} that
$$ \bbK_n(Z)=\bbK_n(\perf(Z))=0 \qquad n <0\,.$$
In particular, if in Proposition~\ref{prop:co-repres-schemes} the base ring $k$ is regular and
noetherian, the negative stable homotopy groups of the spectrum $\bbK(X \times Y)$ vanish.
\end{remark}
}
\subsection{Kontsevich's non-commutative mixed motives}\label{sub:Kontsevich}
Kontsevich introduced in \cite{IAS, Kontsevich-Langlands, Lattice} the category of non-commutative
mixed motives. His construction decomposes in three steps\,:
\begin{itemize}
\item[(1)] First, consider the following category $\Mix_k$, enriched over
symmetric spectra: the objects are the dualizable dg categories (see Definition~\ref{def:saturated});
given saturated dg categories $\cA$ and $\cB$, the symmetric spectrum of
morphisms from $\cA$ to $\cB$ is the non-connective
$K$-theory spectrum $\bbK(\cA^\op\otimes^{\bbL}\cB)$;
the composition corresponds to the derived tensor product of bimodules\footnote{This category is the
non-commutative (and derived) analogue of Grothendieck's category of pure motives\,: $\mathrm{PM}$
stands for Pure Motives, while $\mathrm{K}$ stands for both Kontsevich and $K$-theory.}.

\item[(2)] Then, take the formal triangulated envelope {$\tri(\Mix_k)$} of $\Mix_k$.
Objects in this new category are formal finite extensions of formal shifts of
objects in $\Mix_k$.

\item[(3)] Finally, add formal direct summands for projectors {in $\tri(\Mix_k)$}.
The resulting category $\Mixtr_k$ is what Kontsevich named
the {\em category of non-commutative mixed motives}\footnote{$\mathrm{MM}$
stands for Mixed Motives, while $\mathrm{K}$ stands for both Kontsevich and $K$-theory.}.
\end{itemize}
A precise way to perform these constructions consists on
seeing $\Mixtr_k$ as the spectral category
of perfect $\Mix_k$-modules ($\Mixtr_k$ is the Morita
completion of $\Mix_k$; see \cite[\S 5.2]{MISC} for a precise exposition of these constructions).

Thanks to Theorem~\ref{thm:co-repres-ext}, we are now able to construct a fully-faithful
embedding of $\Mixtr_k$ into our category $\Mloc(e)$ of non-commutative motives,
\ie the base category of the localizing motivator. Note that, in contrast with Kontsevich's
{\em ad hoc} definition, our category of non-commutative motives is defined purely in terms
of precise universal properties.

Let $\cMloc$ be the model category underlying the derivator $\Mloc$.
As, by construction, $\cMloc$ is a left Bousfield localization of
a category of presheaves of symmetric spectra over some small category,
this model category is canonically enriched over symmetric spectra.
We can thus consider the category $\Mloc(e)$ as a category
enriched over symmetric spectra (by considering fibrant and
cofibrant objects in $\cMloc$).

\begin{proposition}\label{prop:Kontsevich}
There is a natural fully-faithful embedding (enriched over symmetric spectra) of Kontsevich's
category of non-commutative motives $\Mixtr_k$ into the
category $\Mloc(e)$. The essential image is the thick triangulated subcategory
spanned by motives of saturated dg categories.
\end{proposition}
\begin{proof}
Given saturated dg categories $\cA$ and $\cB$, Lemma~\ref{lem:tensor-rep} implies
that we have a natural isomorphism in $\Hmo$
$$ \cA^\op \otimes^{\bbL} \cB \simeq \rep(\cA, \cB)\,.$$
Therefore, using Theorem~\ref{thm:co-repres-ext} we obtain a natural fully-faithful
spectral functor
\begin{eqnarray*}
\Mix_k \too \Mloc(e) && \cA \mapsto \Uloc(\cA)\,.
\end{eqnarray*}
By construction of $\Mixtr_k$, \cite[Proposition~5.3.1]{MISC} implies that
this functor extends (uniquely) to a spectral functor
\begin{equation*}
\Mixtr_k \too \Mloc(e)\,.
\end{equation*}
In order to show that this functor is (homotopically) fully-faithful, it is sufficient to prove that
its restriction to a generating family of $\Mixtr_k$ is fully faithful,
which holds by construction.
\end{proof}
{
\subsection{Chern characters}\label{sub:Chern}
In \cite{CT} the authors used the co-representability Theorem~\ref{thm:co-repres} to
classify all natural transformations out of non-connective $K$-theory. More precisely,
they proved in \cite[Theorem~8.1]{CT} that given a localizing invariant $\mathit{L}$,
with values in the derivators of spectra, the data of a natural transformation
$\bbK(-)\Rightarrow \mathit{L}(-)$ is equivalent to the datum of a single class in
the stable homotopy group $\pi_0\mathit{L}(\underline{k})$. From this result they
obtained higher Chern characters (resp. higher trace maps),
from non-connective $K$-theory to (topological) cyclic homology
(resp. to (topological) Hochschild homology); see \cite[Theorem~8.4]{CT}.

However, negative cyclic homology $\mathit{HC}^-$
and periodic cyclic homology $\mathit{HP}$ do not preserve filtered homotopy
colimits since they are defined using infinite products; see \cite[\S5.1]{Loday}.
Therefore, they are not examples of localizing invariants and so the theory developed in \cite{CT}
is not directly applicable in these cases. Nevertheless, we shall explain below why
and how negative cyclic homology and periodic cyclic homology fit naturally in our framework;
see Examples~\ref{ex:negative} and \ref{ex:periodic}.  

Let $\bbD$ be a strong triangulated derivator endowed with a symmetric monoidal structure (with unit ${\bf 1}$) which preserves homotopy colimits in each variable (see \S\ref{sub:M1}), and 
$$ \mathit{E}: \HO(\dgcat) \too \bbD$$
a symmetric monoidal localizing invariant (see Definition~\ref{not:Linv-mon}).
Thanks to Theorem~\ref{thm:main-mon} there is a (unique) symmetric monoidal homotopy
colimit preserving morphism of derivators $\mathit{E}_{\gm}$ which  makes the diagram
$$
\xymatrix{
\HO(\dgcat) \ar[r]^-{\mathit{E}} \ar[d]_-{\Uloc} & \bbD \\
\Mloc \ar[ur]_{\mathit{E}_{\gm}} & 
}
$$
commute (up to unique $2$-isomorphism).
\begin{definition}\label{not:geom-real}
The morphism $\mathit{E}_{\gm}$ is called the {\em geometric realization} of $\mathit{E}$.
Since by hypothesis $\bbD$ is triangulated, we have a natural morphism of derivators
\begin{eqnarray*}\bbR\Hom({\bf 1},-): \bbD \too \HO(\Spt^{\bbN})&& (\mathrm{see}\,\,\S\ref{sub:spectral-enr})\,.
\end{eqnarray*}
The composed morphism
$$ \mathit{E}_{\abs}:= \bbR\Hom({\bf 1}, \mathit{E}_{\gm}(-)): \Mloc \too \HO(\Spt^{\bbN})$$
is called the {\em absolute realization} of $\mathit{E}$.
\end{definition}
Given a symmetric monoidal localizing invariant $\mathit{E}$, we have two objects associated to a non-commutative motive $M \in \Mloc(e)$\,: its geometric realization $\mathit{E}_{\gm}(M)$ and its absolute realization $\mathit{E}_{\abs}(M)$. Although the morphism $\mathit{E}_{\gm}$ always preserves homotopy colimits, this is not always the case for the morphism $\mathit{E}_{\abs}$; a sufficient (and almost necessary) condition for $\mathit{E}_{\abs}$ to preserve homotopy colimits is that the unit ${\bf 1}$ of $\bbD$ is a compact object.
\begin{proposition}\label{prop:Chern}
The geometric realization of $\mathit{E}$ induces a canonical Chern character
\begin{equation*}
\bbK(-) \Rightarrow \bbR\Hom({\bf 1},E(-))\simeq \mathit{E}_{\abs}(\Uloc(-))\,.
\end{equation*} 
Here, $\bbK(-)$ and $\mathit{E}_{\abs}(\Uloc(-))$ are two morphisms of derivators defined on $\HO(\dgcat)$.
\end{proposition}
\begin{proof}
The geometric realization of $\mathit{E}$ is symmetric monoidal and so it sends the unit object $\Uloc(\underline{k})$ to ${\bf 1} \in \bbD$. Therefore, given a small dg category $\cA$, we obtain an induced map
$$\bbK(\cA)\simeq \bbR\Hom_{\Mloc}(\Uloc(\underline{k}), \Uloc(\cA)) \too \bbR\Hom_{\bbD}({\bf 1}, \mathit{E}_{\gm}(\Uloc(\cA)))= \mathit{E}_{\abs}(\Uloc(\cA))\,,$$
where the left-hand side equivalence follows from Theorem~\ref{thm:co-repres}. Since this induced map is functorial in $\cA$, the proof is finished.
\end{proof}
Let us now give some examples which illustrate Proposition~\ref{prop:Chern}.

\begin{example}[Non-connective $K$-theory]
The tautological version of the situation above is:
for $E=\Uloc$, $E_{\gm}$ is by definition the identity of $\Mloc$
(see Theorem~\ref{thm:main-mon}), while $E_{\abs}=\bbK$ is non-connective $K$-theory (see Theorem~\ref{thm:co-repres}).
The corresponding Chern character is the identity, and this is in this precise
sense that non-connective $K$-theory is initial among \emph{absolute} homology theories.
\end{example}

\begin{example}[Hochschild homology]
Take for $\mathit{E}$ the symmetric monoidal localizing invariant
$$\mathit{HH}: \HO(\dgcat) \too \HO(\cC(k))$$
of Example~\ref{ex:HH}. In this case, there is no difference (up to the Dold-Kan correspondance relating complexes of $k$-modules and spectra) between the geometric and the absolute realization: if we consider $\HO(\cC(k))$ as enriched over itself, then the morphism
$$\bbR\Hom(k,-): \HO(\cC(k)) \too \HO(\cC(k))$$ is (isomorphic to) the identity. Therefore
by Proposition~\ref{prop:Chern}, we obtain a canonical Chern character
$$ \bbK(-) \Rightarrow \mathit{HH}(-)\,.$$
\end{example}
\begin{example}[Negative cyclic homology]\label{ex:negative}
Take for $\mathit{E}$ the symmetric monoidal localizing invariant
$$ \mathit{C}: \HO(\dgcat) \too \HO(\cC(\Lambda))$$
of Example~\ref{ex:Mixedcomp}. Given a small dg category $\cA$, we have an equivalence
$$ \mathit{C}_{\abs}(\Uloc(\cA)) = \bbR\Hom(k, \mathit{C}(\cA)) \simeq \mathit{HC}^-(\cA)\,,$$
where $\mathit{HC}^-(\cA)$ denotes the negative cyclic homology complex
of $\cA$; see \cite[\S2.2]{Exact1}. Therefore by Proposition~\ref{prop:Chern}, we
obtain a canonical Chern character
$$ \bbK(-) \Rightarrow \mathit{HC}^-(-)\,.$$
\end{example}
\begin{example}[Periodic cyclic homology]\label{ex:periodic}
Take for $\mathit{E}$ the symmetric monoidal localizing invariant
$$ (\mathit{P}\circ \mathit{C}): \HO(\dgcat) \too \HO(k[u]\text{-}\mathrm{Comod})$$
of Example~\ref{ex:Periodiccomp}. Assuming that the ground ring $k$ is a field,
for any small dg category $\cA$, we have a natural identification
$$ (\mathit{P}\circ \mathit{C})_{\abs}(\Uloc(\cA)) = \bbR\Hom(k[u], (\mathit{P}\circ\mathit{C})(\cA))
\simeq \mathit{HP}(\cA)\,,$$
where $\mathit{HP}(\cA)$ denotes the periodic cyclic homology complex of $\cA$.
This can be seen as follows.
Given a mixed complex $M$ (see Example~\ref{ex:Mixedcomp}), a map $k[u]\to P(M)$
in $k[u]\text{-}\mathrm{Comod}$ corresponds to a collection
of maps $k\to {(M\otimes^\bbL_\Lambda k)}[2n]$, $n\geq 0$, in $\cC(k)$ which are
compatible with the operator $S$. In other words, these data correspond to a map in $\cC(k)$
from $k$ to the tower
$$\cdots \xrightarrow{\, S \, } {(M\otimes^\bbL_\Lambda k)}[-2n] \xrightarrow{\, S \, }
{(M\otimes^\bbL_\Lambda k)}[-2n+2]  \xrightarrow{\, S \, } \cdots
 \xrightarrow{\, S \, } (M\otimes^\bbL_\Lambda k)[-2]  \xrightarrow{\, S \, } M\otimes^\bbL_\Lambda k\, .$$
In other words, we have\,:
$$\bbR\Hom(k[u],P(M))\simeq \underset{n}{\mathrm{holim}}\,(M\otimes^\bbL_\Lambda k)[-2n]\, .$$
The Milnor short exact sequence \cite[Proposition~7.3.2]{Hovey} applied to this
homotopy limit corresponds to the short exact sequence
$$0\too {\varprojlim_n}^1 H_{i+2n-1}(M\otimes^\bbL_\Lambda k)
\too \Hom(k[u],P(M)[-i])\too  {\varprojlim_n}\,  H_{i+2n}(M\otimes^\bbL_\Lambda k)
\too 0\, .$$
Now, let $\cA$ be a small dg category. The above arguments, with $M=\mathit{C}(\cA)$, allow
us to deduce the  formula
$$ \bbR\Hom(k[u],{(P\circ C)}(\cA)) \simeq
\underset{n}{\mathrm{holim}}\,(\mathit{C}(\cA)\otimes^{\bbL}_{\Lambda}k)[-2n] \simeq
\mathit{HP}(\cA)\,.$$
Therefore, by Proposition~\ref{prop:Chern}, we obtain a canonical Chern character
$$ \bbK(-) \Rightarrow \mathit{HP}(-)\,.$$
\end{example}
}
\subsection{To{\"e}n's secondary $K$-theory}\label{sub:secondary}
To{\"e}n introduced in \cite{Sedano, NW-Toen} a ``categorified'' version of algebraic
$K$-theory named {\em secondary $K$-theory}; see \cite{Abel} for a survey article. 
\begin{definition}[{\cite[\S5.4]{Sedano}}]\label{def:secondary}
Given a commutative ring $k$, let $\bbZ[\Hmo_{\mathsf{sat},k}]$ be the free
abelian group on the isomorphism classes of objects in $\Hmo_{\mathsf{sat},k}$
(see Notation~\ref{not:sat}). The {\em secondary $K$-theory group $K^{(2)}_0(k)$ of $k$}
is the quotient of $\bbZ[\Hmo_{\mathsf{sat},k}]$ by the relations $[\cB] = [\cA] + [\cC]$
associated to exact sequences (see Definition~\ref{def:ses})
$$ \cA \too \cB \too \cC$$
of saturated dg categories.
\end{definition}
\begin{remark}
\begin{itemize}
\item[(i)] Thanks to Theorem~\ref{thm:Toen-dualizable} the category $\Hmo_{\mathsf{sat},k}$
coincides with the category of dualizable objects in $\Hmo_k$.
Therefore, by Remark~\ref{rk:adj-dualizable}(iii), the derived tensor product in $\Hmo_k$
restricts to a bifunctor
\begin{equation*}
- \otimes^{\bbL}-: \Hmo_{\mathsf{sat},k} \times \Hmo_{\mathsf{sat},k} \too \Hmo_{\mathsf{sat},k}\,.
\end{equation*}
By \cite[Proposition~1.6.3]{Drinfeld} the derived tensor product preserves exact sequences
(in both variables), and so we obtain a commutative ring structure on $K^{(2)}_0(k)$.
\item[(ii)] Given a ring homomorphism $k \to k'$, we have a derived base change functor
\begin{eqnarray*}
 - \otimes^{\bbL}_k k': \Hmo_k \too \Hmo_{k'} && \cA \mapsto \cA \otimes^{\bbL}_k k'\,.
\end{eqnarray*} 
This functor preserves exact sequences and is symmetric monoidal. Therefore, by
Theorem~\ref{thm:Toen-dualizable} and Remark~\ref{rk:adj-dualizable}(iv), we
obtain a ring homomorphism
$$ K^{(2)}_0(k) \too K_0^{(2)}(k')\,.$$
\end{itemize}
In conclusion, secondary $K$-theory is a functor $K_0^{(2)}(-)$ from the category of commutative
rings to itself.
\end{remark}
One of the motivations for the study of this secondary $K$-theory was its expected connection with
an hypothetical Grothendieck ring of motives in the non-commutative setting; see \cite[page~1]{NW-Toen}.
Thanks to Theorem~\ref{thm:main-mon}, we are now able to make this connection precise;
see Remarks~\ref{rk:non-trivial}-\ref{rk:connection}. We shall use the following well known property of dualizable objects in a
triangulated category.

\begin{proposition}\label{prop:thick}
Let $\cC$ be a closed symmetric monoidal triangulated category. Then, the category $\cC^{\vee}$ of
dualizable objects in $\cC$ (see Definition~\ref{def:dualizable}) is a symmetric monoidal thick
triangulated subcategory of $\cC$. 
\end{proposition}

\begin{proof}
The fact that dualizable objects are stable under tensor product
is clear; see Remark~\ref{rk:adj-dualizable}(iii). For an object $X$ in $\cC$, set $X^\vee=\uHom(X,{\bf 1})$. Given two objects $X$ and $Y$ in $\cC$, we have a canonical map
$$u_{X,Y}:X^\vee\otimes Y\too \uHom(X,Y)$$
which corresponds by adjunction to the map
$X^\vee\otimes X\otimes Y\to Y$ obtained by tensoring $Y$ with the
evaluation map $X\otimes X^\vee \to {\bf 1}$.
The object $X$ is dualizable if and only if the map $u_{X,Y}$
is invertible for any object $Y$. Since for any fixed object $Y$ the map $u_{X,Y}$
is a natural transformation of triangulated functors, we conclude that $\cC^\vee$ is a thick triangulated subcategory of $\cC$.
\end{proof}

\begin{notation}\label{not:thick}
Thanks to Theorem~\ref{thm:main-mon} the localizing motivator carries a symmetric monoidal
structure, and so its base category $\Mot^{\loc}_{\dg,k}(e)$ is a symmetric monoidal triangulated category.
Therefore by Proposition~\ref{prop:thick}, the category $\Mot^{\loc}_{\dg,k}(e)^{\vee}$ of dualizable objects is {a
symmetric monoidal thick triangulated subcategory of $\Mot^{\loc}_{\dg,k}(e)$}.

Let $\Mixtr_k$ be Kontsevich's category of non-commutative mixed motives. Thanks to Proposition~\ref{prop:Kontsevich}, we can identify it with the thick triangulated subcategory of {$\Mot^{\loc}_{\dg,k}(e)^{\vee}$}
generated by objects of shape $\Uloc(\cA)$, where $\cA$ runs over the family of
saturated dg categories over $k$. Therefore, $\Mixtr_k$ is naturally a rigid symmetric monoidal
triangulated category.
\end{notation}

\begin{definition}\label{def:motivic-ring} Let $k$ be a commutative ring. The {\em Grothendieck
ring $\cK_0(k)$ of non-commutative motives over $k$} is the Grothendieck ring $K_0(\Mixtr_k)$.
\end{definition}
{
\begin{remark}[Non-triviality]\label{rk:non-trivial}
Recall from Example~\ref{ex:HH} the construction of the symmetric monoidal localizing invariant
$$ \mathit{HH}: \HO(\dgcat) \too \HO(\cC(k))\,.$$
By restricting its geometric realization (see Definition~\ref{not:geom-real}) to the base category, we obtain a symmetric monoidal triangulated functor
$$ \mathit{HH}_{\gm}(e): \Mot^{\loc}_{\dg,k}(e) \too  \HO(\cC(k))(e)=\cD(k)\,.$$
Recall that the dualizable objects in $\cD(k)$ are precisely the perfect complexes of $k$-modules. Therefore, by Remark~\ref{rk:adj-dualizable}~(iv), $\mathit{HH}_{\gm}(e)$
sends dualizable objects to perfect complexes and so it induces a ring homomorphism
\begin{equation}\label{eq:rank}
 r\cK_0: \cK_0(k)=K_0(\Mixtr_k) \too K_0(\cD_c(k))=K_0(k)\,.
\end{equation}
Finally, since $K_0(k)$ is non-trivial we conclude that $\cK_0(k)$ is also non-trivial.
\end{remark}
}
\begin{remark}[Functoriality]\label{rk:functoriality}
Given a ring homomorphism $k \to k'$, we have a base change functor
\begin{eqnarray*}
- \otimes_k k' : \dgcat_k \too \dgcat_{k'} && \cA \mapsto \cA \otimes_k k'\,.
\end{eqnarray*}
This functor gives rise to a morphism of derivators
$$ - \otimes^{\bbL}_k k' : \HO(\dgcat_k) \too \HO(\dgcat_{k'})\,,$$
which is symmetric monoidal, preserves homotopy colimits (and the point), and satisfies localization;
see Theorem~\ref{thm:Uloc}. Therefore, the composition
$$ \HO(\dgcat_k) \stackrel{-\otimes^{\bbL}_k k'}{\too} \HO(\dgcat_{k'}) \stackrel{\Uloc}{\too}
\Mot^{\loc}_{\dg,k'}$$
is a symmetric monoidal localizing invariant; see \ref{not:Linv-mon}.
Using Theorem~\ref{thm:main-mon},
we obtain a (unique) symmetric monoidal morphism, which we still denoted by $-\otimes^{\bbL}_k k'$,
making the diagram
 \begin{equation}\label{eq:base-change}\begin{split}
\xymatrix{
\HO(\dgcat_k) \ar[d]_{\Uloc} \ar[rr]^{-\otimes^{\bbL}_k k'} && \HO(\dgcat_{k'}) \ar[d]^{\Uloc} \\
\Mot^{\loc}_{\dg,k} \ar[rr]_{-\otimes^{\bbL}_k k'} && \Mot^{\loc}_{\dg,k'} \\
}\end{split}
\end{equation}
commute (up to $2$-isomorphism).
By restricting ourselves to the base categories, we have a symmetric monoidal triangulated functor
$$-\otimes^{\bbL}_k k': \Mot^{\loc}_{\dg,k}(e) \too \Mot^{\loc}_{\dg,k'}(e)\, .$$
As the (derived) change of scalars functor preserves saturated dg categories, we obtain then an induced ring homomorphism 
$$\cK_0(k) \too \cK_0(k')\,.$$
In conclusion, the Grothendieck ring of non-commutative motives is a functor $\cK_0(-)$ from the category
of commutative rings to itself.
\end{remark}
\begin{remark}[Connection]\label{rk:connection} Thanks to Theorem~\ref{thm:main-mon}, the functor
$$ \Uloc: \Hmo_k \too \Mot^{\loc}_{\dg,k}(e)$$
is symmetric monoidal. By construction it sends exact sequences to distinguished triangles, and so it induces a ring homomorphism
\begin{equation}\label{eq:ring} 
\Phi(k): K_0^{(2)}(k) \too \cK_0(k)\,.
\end{equation}
{Note that this ring homomorphism is not necessarily surjective because of step (3) in the construction of $\Mixtr_k$. However, the image of $\Phi(k)$ can be described as the Grothendieck group of the triangulated category $\tri(\Mix_k)$\,: by cofinality the Grothendieck ring $K_0(\tri(\Mix_k))$ is a subring of $\cK_0(k)$ and by d{\'e}vissage $\Phi(k)$ surjects on $K_0(\tri(\Mix_k))$.} 
Moreover, the above (up to $2$-isomorphism) commutative square (\ref{eq:base-change}) shows us that the ring homomorphism (\ref{eq:ring}) gives rise to a natural transformation of functors
\begin{eqnarray*}
K_0^{(2)}(-) \Rightarrow \cK_0(-)\,, && k \mapsto \Phi(k)\,.
\end{eqnarray*}\label{rem:phiksurj}
Now, let $R$ be a commutative ring and $l: K_0^{(2)}(k) \to R$ a {\em realization of $K_0^{(2)}(k)$},
\ie a ring homomorphism. Note that if there exists a symmetric monoidal localizing invariant
$$ \HO(\dgcat_k) \too \bbD\,,$$
whose induced ring homomorphism (see Proposition~\ref{prop:thick})
$$ K_0^{(2)}(k) \too K_0(\bbD(e)^{\vee})$$
identifies with $l$, then $l$ factors through $\Phi(k)$. An interesting example is proved by To{\"e}n's
rank map (see \cite[\S5.4]{Sedano})
$$ rk_0: K_0^{(2)} \too K_0(k)\,.$$
Thanks to \cite[\S5.4 Lemma~3]{Sedano} this rank map is induced from the symmetric monoidal localizing invariant 
$$\mathit{HH}: \HO(\dgcat_k) \too \HO(\cC(k))$$
of Example~\ref{ex:HH}. Therefore, it corresponds to the following composition
$$ K_0^{(2)}(k) \stackrel{\Phi(k)}{\too} \cK_0(k) \stackrel{r\cK_0}{\too} K_0(k)\,,$$
where $r\cK_0$ is the ring homomorphism (\ref{eq:rank}) of Remark~\ref{rk:non-trivial}.  
\end{remark}



\subsection{Euler characteristic}
\begin{definition}\label{def:Euler}
Let $\cC$ be a symmetric monoidal category with monoidal product $\otimes$ and
unit object ${\bf 1}$. Given a dualizable object $X$ in $\cC$ (see Definition~\ref{def:dualizable})
its {\em Euler characteristic $\chi(X)$} is the following composition
$$ \chi(X): {\bf 1} \stackrel{\delta}{\too} X^{\vee}  \otimes X \stackrel{\tau}{\too} X
\otimes X^{\vee} \stackrel{\mathrm{ev}}{\too} {\bf 1}\,,$$
where $\tau$ denotes the symmetry isomorphism.
\end{definition}

\begin{remark}\label{Eulerringmap}
Let $\cC$ is a well behaved symmetric monoidal triangulated category; e.g. $\cC=\bbD(e)$ for some symmetric monoidal triangulated derivator $\bbD$. Then, thanks to \cite[Theorem~1.9]{Maytraces}, the Euler characteristic gives rise to a ring homomorphism
$$\chi:K_0(\cC^\vee)\too \Hom_\cC({\bf 1},{\bf 1})\, .$$
\end{remark}

\begin{proposition}
Let $\cA$ be a saturated dg category.
Then its Euler characteristic $\chi(\cA)$ in $\Hmo$ is the isomorphism
class of $\cD_c(k)$ which is associated to the (perfect)
Hochschild homology complex $\mathit{HH}(\cA)$ of $\cA$
(see Example~\ref{ex:HH}). 
\end{proposition}

\begin{proof}
From Theorem~\ref{thm:Toen-dualizable} (and its proof) we see
that the dual of $\cA$ is its opposite dg category $\cA^\op$,
and that the following composition in $\Hmo$
$$ \chi(\cA): \underline{k} \stackrel{[\cA(-,-)]}{\too} \cA^\op\otimes^{\bbL}\cA \stackrel{\tau}{\too}
\cA\otimes^{\bbL}\cA^\op \stackrel{[\cA(-,-)]}{\too} \underline{k}$$
corresponds to the complex
$$ \cA(-,-) \bigotimes_{\cA^\op\otimes^{\bbL}\cA}^{\bbL} \cA(-,-)\,.$$
By \cite[Proposition~1.1.13]{Loday} this complex of $k$-modules computes Hochschild
homology of $\cA$ (with coefficients in itself), which achieves the proof.
\end{proof}

\begin{proposition}\label{prop:preserve-Euler}
Let $F: \cC \to \cC'$ be a symmetric monoidal functor between symmetric monoidal
categories with unit objects ${\bf 1}$ and ${\bf 1'}$.
Then, given a dualizable object $X$ in $\cC$, the Euler characteristic $\chi(F(X))$ of $F(X)$
agrees with $F(\chi(X))$ on ${\bf 1'} \simeq F({\bf 1})$.
\end{proposition}

\begin{proof}
It is a straightforward consequence of the definitions. The details are left as an exercise for the reader.
\end{proof}

\begin{proposition}\label{prop:Euler}
Let $\cA$ be a saturated dg category. Then, $\chi(\Uloc(\cA))$ is the element of the
Grothendieck group $K_0(k)$ which is associated to the (perfect)
Hochschild homology complex $\mathit{HH}(\cA)$ of $\cA$.
\end{proposition}

\begin{proof}
Thanks to Theorem~\ref{thm:main-mon}, the universal localizing invariant
$\Uloc$ is symmetric monoidal. Using Theorem~\ref{thm:co-repres},
we see that the map
$$ \Uloc(e): \mathrm{Iso}\,\cD_c(k)\simeq\Hom_{\Hmo}(\underline{k},\underline{k})
\too \Hom(\Uloc(\underline{k}),\, \Uloc(\underline{k}))\simeq K_0(k) $$
sends an element in $\mathrm{Iso}\,\cD_c(k)$ to the corresponding class in the Grothendieck
group $K_0(\cD_c(k))\simeq K_0(k)$.
Hence, Theorem~\ref{thm:Toen-dualizable}
and Proposition~\ref{prop:preserve-Euler} achieve the proof.
\end{proof}


\begin{example}
Recall from Example~\ref{ex:Toen} that, given a smooth and proper $k$-scheme $X$,
we have a saturated dg category $\perf(X)$ which enhances the category of compact objects
in $\cD_{qcoh}(X)$. Thanks to Keller~\cite{Exact, Exact1} the Hochschild homology of $\perf(X)$
(see Example~\ref{ex:HH}) agrees with the Hochschild homology of $X$
in the sense of Weibel~\cite{Weibel-Schemes}.
Therefore, by Proposition~\ref{prop:Euler} the Euler characteristic of $\Uloc(\perf(X))$ is the element
of the Grothendieck group $K_0(k)$ which is associated to the (perfect)
Hochschild homology complex $\mathit{HH}(X)$ of $X$. 

When $k$ is the field of complex numbers, the Grothendieck ring $K_0(\bbC)$ is
naturally isomorphic to $\bbZ$ and the Hochschild homology of $X$ agrees with the
Hodge cohomology $H^{\ast}(X, \Omega^{\ast}_X)$ of $X$.
Therefore, when we work over $\bbC$, the Euler characteristic of $\Uloc(\perf(X))$
is the classical Euler characteristic of $X$.
\end{example}
\appendix
\section{Grothendieck derivators}\label{appendix:A}
The original reference for the theory of derivators is Grothendieck's
manus\-cript~\cite{Grothendieck} and Heller's monograph~\cite{Heller0}.
See also \cite{imdir,propuni,CN,Duke}.
\subsection{Prederivators}\label{sub:prederivators}
A {\em prederivator} $\bbD$ consists of a strict contravariant $2$-functor from the $2$-category of small
categories to the $2$-category of categories 
$$
\bbD: \Cat^{\op} \longrightarrow \CAT.
$$
Prederivators organize themselves naturally in a $2$-category: the $1$-morphisms
(usually called morphisms) are the pseudo natural transformations and the $2$-morphisms
are the modifications; see \cite[\S 5]{CN} for details.
Given prederivators $\bbD$ and $\bbD'$, we denote by $\uHom(\bbD, \bbD')$ the category of morphisms.

Given a category $\cM$, we denote by $\underline{\cM}$ the prederivator defined for every small
category $X$ by
$$ \underline{\cM}(X) := \Fun(X^\op, \cM),$$
where $\Fun(X^\op, \cM)$ is the category of presheaves on $X$ with values in $\cM$.
If $\cW$ is a class of morphisms in $\cM$, we denote by
$\underline{\cM}[\cW^{-1}]$ the prederivator defined for every small category $X$ by
$$ \underline{\cM}[\cW^{-1}](X) := \Fun(X^\op, \cM)[\cW^{-1}].$$
Here $\Fun(X^\op, \cM)[\cW^{-1}]$ is the localization of $\Fun(X^\op, \cM)$ with respect to the
class of morphism which belong termwise to $\cW$. Note that the assignment
$(\cM, \cW) \mapsto \underline{\cM}[\cW^{-1}]$
is {\em $2$-functorial}, \ie given a natural transformation
$$
\xymatrix{
(\cM, \cW) \ar@/^1pc/[rr]^F \ar@/_1pc/[rr]_G &\Downarrow& (\cN, \cV)
}
$$
between functors, such that $F(\cW)\subset (\cV)$ and $G(\cW) \subset (\cV)$, we obtain an induced $2$-morphism
$$
\xymatrix{
\underline{\cM}[\cW^{-1}] \ar@/^1pc/[rr]^F \ar@/_1pc/[rr]_G &\Downarrow& \underline{\cN}[\cV^{-1}]
}
$$
of prederivators.
\subsection{Derivators}\label{sub:derivators}
A {\em derivator} is a prederivator which is subject to certain conditions, the main ones being that for any
functor $u:X\to Y$ between small categories, the \emph{inverse image
functor}
$$u^*=\bbD(u):\bbD(Y)\too\bbD(X)$$
has a left adjoint, called the \emph{homological direct image functor},
$$u_!:\bbD(X)\too\bbD(Y)\, ,$$
as well as right adjoint, called the \emph{cohomological direct image functor}
$$u_{\ast}:\bbD(X)\too\bbD(Y)\,.$$
See~\cite{imdir} for details. Similarly to the case of prederivators, derivators organize themselves
in a $2$-category. Given derivators $\bbD$ and $\bbD'$, we denote by $\uHom_{\flt}(\bbD,\bbD')$ the
category of morphisms of derivators which preserve filtered homotopy colimits,
and by $\HomC(\bbD,\bbD')$ the
category of morphisms of derivators which commute with all homotopy colimits;
see \cite{imdir,propuni}.
 
The essential example of a derivator to keep in mind is the derivator $\bbD=\HO(\cM)$
associated to a (complete and cocomplete) Quillen model category~$\cM$
(see \cite[Theorem~6.11]{imdir}), which is defined for every small category~$X$ by
\begin{equation*}
\HO(\cM)(X):=\Ho\big(\Fun(X^{\op},\cM)\big)\,.
\end{equation*}
In this case, any colimit (resp. limit)
preserving left (resp. right) Quillen functor induces a morphism of derivators
which preserves homotopy colimits (resp. limits); see \cite[Proposition\, 6.12]{imdir}.

Finally, we denote by $e$ the $1$-point category with one object and one
(identity) morphism. Heuristically, the category $\bbD(e)$ is the
basic ``derived" category under consideration in the
derivator~$\bbD$. For instance, if $\bbD=\HO(\cM)$ then
$\bbD(e)=\Ho(\cM)$ is the usual homotopy category of $\cM$.
\subsection{Properties}\label{sub:properties}
\begin{itemize}
\item[(i)] A derivator $\bbD$ is called {\em strong} if for every finite free category
$X$ and every small category $Y$, the natural functor $ \bbD(X
\times Y) \to \Fun(X^{\op},\bbD(Y))$ is full and
essentially surjective.
\item[(ii)]
A derivator $\bbD$ is called {\em regular} if sequential homotopy
colimits commute with finite products and homotopy pullbacks.
\item[(iii)] A derivator $\bbD$ is called {\em pointed} if for any closed immersion $i:Z
\rightarrow X$ in $\Cat$ the cohomological direct image functor
$i_{\ast}:\bbD(Z) \to \bbD(X)$ has a right adjoint, and if, dually,
for any open immersion $j:U \to X$
the homological direct image functor $ j_!: \bbD(U) \to
\bbD(X)$ has a left adjoint; see~\cite[Definition~1.13]{CN}.
\item[(iv)]
A derivator $\bbD$ is called {\em triangulated} or {\em stable} if it is
pointed and if every global commutative square is
cartesian exactly when it is cocartesian; see~\cite[Definition~1.15]{CN}.
\end{itemize}
A strong derivator is the same thing as a small homotopy theory in
the sense of Heller~\cite{Heller}. Thanks to~\cite[Proposition~2.15]{catder}, if
$\cM$ is a Quillen model category its associated derivator
$\HO(\mathcal{M})$ is strong. Moreover, if sequential homotopy
colimits commute with finite products and homotopy pullbacks
in~$\cM$, the associated derivator $\HO(\cM)$ is regular. Notice that if
$\cM$ is pointed, then the derivator $\HO(\cM)$ is pointed. Finally, a pointed
Quillen model category $\cM$ is stable
if and only if its associated derivator $\HO(\cM)$ is triangulated.

\subsection{Kan extensions}\label{sub:Kanext}
Given a small category $A$, we denote by $\mathsf{Hot}_A=\HO(\spref A)$
the derivator associated to the projective model category structure on the
category of simplicial presehaves. We then have a Yoneda
embedding
\begin{equation}\label{YonedaHot}
h:\underline{A}\too \mathsf{Hot}_A\, .
\end{equation}
Let $\bbD$ be a derivator.
A $2$-functorial version of the Yoneda lemma
gives a canonical equivalence of categories
\begin{equation*}
\uHom(\underline{A},\bbD) \simeq \bbD(A^\op)\, .
\end{equation*}

\begin{theorem}\label{derKanext1}
The morphism of prederivators \eqref{YonedaHot} is the universal
morphism from $\underline{A}$ to a derivator. In other words,
given any derivator $\bbD$, the induced functor
$$h^*: \uHom_!(\mathsf{Hot}_A,\bbD) \stackrel{\sim}{\too} \uHom(\underline{A},\bbD)$$
is an equivalence of categories.
\end{theorem}

\begin{proof}
See \cite[Corollaire~3.26]{propuni}.
\end{proof}
\subsection{Monoidal structures}\label{sub:M1}
Thanks to \cite[Proposition\,5.2]{propuni} the $2$-category of prederivators
form a closed symmetric monoidal $2$-category with respect to the cartesian product.
Given two prederivators $\bbD$ and $\bbD'$, we denote by $\mathbf{Hom}(\bbD,\bbD')$
the corresponding internal Hom; see \cite[\S5.1]{propuni}.

Given a prederivator $\bbD$, by a {\em symmetric monoidal structure on $\bbD$} we
mean a structure of symmetric pseudo monoid on $\bbD$.
In other words, for every small category $X$, $\bbD(X)$ is a
symmetric monoidal category, and for every functor $u:X \to Y$ between small categories,
the inverse image functor
$$ u^{\ast}: \bbD(u): \bbD(Y) \too \bbD(X)$$
is symmetric monoidal; see \cite[\S 5.4]{propuni}.
A {\em symmetric monoidal prederivator} is a prederivator endowed
with a symmetric monoidal structure.
Given symmetric monoidal prederivators $\bbD$ and $\bbD'$, we
denote by $\uHom^{\otimes}(\bbD, \bbD')$ the category of symmetric monoidal
morphisms; see \cite[\S 5.11]{propuni} for details.
A \emph{symmetric monoidal derivator}
is a symmetric monoidal prederivator $\bbD$ which is also
a derivator, and such that the tensor product {\em preserves homotopy colimits
in each variable}, \ie sucht that, for any object $X \in \bbD(e)$ the induced morphism
$$X \otimes - : \bbD \too \bbD$$
preserves homotopy colimits.

Given symmetric monoidal derivators $\bbD$ and $\bbD'$, we
denote by $\uHom^{\otimes}_!(\bbD, \bbD')$ the category of symmetric monoidal
morphisms which preserve homotopy colimits.

A basic example of a symmetric monoidal prederivator is given as follows\,:
let $\cM$ be a symmetric monoidal category (with monoidal product $-\otimes-$)
and $\cW$ a class of morphisms in $\cM$. If the monoidal product preserves the
class $\cW$, \ie if we have an inclusion $\cW \otimes \cW \subseteq \cW$,
then the prederivator $\underline{\cM}[\cW^{-1}]$ of \S \ref{sub:prederivators}
is naturally a symmetric monoidal prederivator.
Moreover, if $F:(\cM,\cW) \to (\cN,\cV)$ is a
symmetric monoidal functor such that $F(\cW)\subset \cV$, then the induced morphism
$$ \underline{\cM}[\cW^{-1}] \too \underline{\cN}[\cV^{-1}]$$
is symmetric monoidal.

As for examples of symmetric monoidal derivators, most of them are obtained
from symmetric monoidal model categories \cite[Definition~4.2.6]{Hovey}.

\begin{proposition}\label{prop:Mderivator}
Let $\cM$ be a symmetric monoidal model category.
Then its associated derivator $\HO(\cM)$ carries a symmetric monoidal structure.
Moreover, any symmetric monoidal left Quillen functor
between symmetric monoidal model categories induces a symmetric monoidal morphism
between the associated derivators.
\end{proposition}

\begin{proof}
See \cite[Proposition~6.1]{propuni}.
\end{proof}

\subsection{Derived Day convolution product}\label{sub:biKanext}
Let $A$ be a small symmetric monoidal category. Then $\underline{A}$
is a symmetric monoidal prederivator.
Moreover, as $\spref A$ is then a symmetric monoidal model
category (see Theorem~\ref{thm:Mprojective}), the derivator $\mathsf{Hot}_A$ is then
a symmetric monoidal derivator, in such a way that the Yoneda embedding
\eqref{YonedaHot} is a symmetric monoidal morphism of prederivators.
Given a symmetric monoidal derivator $\bbD$, we thus have an induced functor
\begin{equation}\label{restrictyonedamonoidal}
h^*: \uHom^\otimes_!(\mathsf{Hot}_A,\bbD)\too\uHom^\otimes(\underline{A},\bbD)\, .
\end{equation}
\begin{theorem}\label{monoidalderKanext}
The functor \eqref{restrictyonedamonoidal} is an equivalence of categories.
\end{theorem}

\begin{proof}
This is simply a variation on Theorem \ref{derKanext1}.
Given two derivators $\bbD'$ and $\bbD''$, the derivator $\mathbf{Hom}_!(\bbD',\bbD'')$
is defined by
$$\mathbf{Hom}_!(\bbD',\bbD'')(A)=\uHom_!(\bbD',\bbD''_A)\, ,$$
where $\bbD''_A$ is in turn the derivator of ``presheaves on $A$ with values
in $A$'', \ie the derivator defined by
$$\bbD''_A(X)=\bbD''(A\times X)\,.$$
For a third derivator $\bbD$, the data of a morphism of prederivators
$$\bbD\times\bbD'\too\bbD''$$
which preserves homotopy colimits in each variable is equivalent to the data
of an object in $\uHom_!(\bbD,\mathbf{Hom}_!(\bbD',\bbD''))$; see \cite[Lemme~5.18]{propuni}.
Moreover, when $\bbD'$ is of the form $\mathsf{Hot}_B$, with $B$ a small category, it follows immediately from Theorem \ref{derKanext1} that
we have equivalences of derivators
$$\mathbf{Hom}_!(\mathsf{Hot}_B,\bbD'')\simeq \mathbf{Hom}(\underline{B},\bbD'')
\simeq \bbD''_{B^\op}\, .$$
Hence, if $A$ is another small category, we obtain canonical
equivalences of categories\,:
$$\begin{aligned}
\uHom_!(\mathsf{Hot}_A,\mathbf{Hom}_!(\mathsf{Hot}_B,\bbD''))
& \simeq \uHom(\underline{A},\mathbf{Hom}_!(\mathsf{Hot}_B,\bbD''))\\
& \simeq \uHom(\underline{A},\mathbf{Hom}(\underline{B},\bbD''))\\
& \simeq \uHom(\underline{A}\times \underline{B},\bbD'')=\bbD''(A^\op\times B^\op)\, .
\end{aligned}$$
For $A=B$ and $\bbD''=\mathsf{Hot}_A$, we note that the tensor product
on $\mathsf{Hot}_A$ corresponds, under these equivalences of categories, to the tensor product $\otimes : A\times A\too A$ composed with the
Yoneda embeding \eqref{YonedaHot}.

More generally, we obtain (by induction on $n\geq 0$) that for any 
$n$-tuple of small categories $(A_1,\ldots,A_n)$, the category
of morphisms $\mathsf{Hot}_{A_1}\times \dots \times \mathsf{Hot}_{A_n}\too\bbD''$
which preserve homotopy colimits in each variable is
canonically equivalent to the category of morphisms
$\underline{A_1}\times \dots \times \underline{A_n}\too\bbD''$.
This fact implies that the symmetric monoidal structure on $A$
extends uniquely to a symmetric monoidal structure on the derivator
$\mathsf{Hot}_A$. Moreover, the category of
symmetric monoidal morphisms from $\mathsf{Hot}_A$ to $\bbD''$, which preserve homotopy colimits, is canonically equivalent
to the category of symmetric monoidal morphisms from $\underline{A}$ to
$\bbD''$.
\end{proof}

\subsection{Left Bousfield localization}\label{sub:leftBousfield}
Let $\bbD$ be a derivator and $S$ a class of morphisms in the base category $\bbD(e)$.
We say that the derivator $\bbD$ admits a {\em left Bousfield localization} with
respect to the class $S$, if there exists a morphism of derivators
$$ \gamma : \bbD \too \Loc_S\bbD\,,$$
which commutes with homotopy colimits, sends the elements of $S$ to
isomorphisms in $\Loc_S\bbD(e)$, and satisfies the following
universal property\,: given any derivator $\bbD'$, the morphism
$\gamma$ induces an equivalence of categories
\begin{equation*}
\gamma^{\ast}: \HomC(\Loc_S\bbD,\bbD')
  \stackrel{\sim}{\too}
  \uHom_{!,S}(\bbD,\bbD')\,,
\end{equation*}
where $\uHom_{!,S}(\bbD,\bbD')$ denotes
the category of morphisms of derivators which commute with homotopy
colimits and send the elements of $S$ to isomorphisms in $\bbD'(e)$.
\begin{theorem}\label{thm:Quillenloc}
Let $\cM$ be a left proper cellular model category and $S$ a set of maps in the
homotopy category $\Ho(\cM)$ of $\cM$. Consider the left Bousfield localization
$\Loc_S\cM$ of $\cM$  with respect to the set $S$, $\ie$ to perform the
localization we choose in $\cM$ a representative for each element of $S$.
Then, the induced morphism of
derivators $\HO(\cM)\to \HO(\Loc_S\cM)$ is a left Bousfield
localization of the derivator $\HO(\cM)$ with respect to the set $S$.
Moreover, we have a natural adjunction of derivators\,:
$$
\xymatrix{
\HO(\cM) \ar@<-1ex>[d] \\
\HO(\Loc_S\cM) \ar@<-1ex>[u]\,.
}
$$
\end{theorem}

\begin{proof}
See \cite[Theorem~4.4]{Duke}.
\end{proof}

\begin{remark}\label{rk:regular}
If the domains and codomains of the elements of the set $S$
are homotopically finitely presented (see Definition~\ref{def:HFP}),
the morphism $\HO(\Loc_S\cM) \to \HO(\cM)$ (right adjoint to the localizing functor)
preserves filtered homotopy colimits. Therefore, under these hypothesis, if $\HO(\cM)$ is
regular so it is $\HO(\Loc_S\cM)$.
\end{remark}
By~\cite[Lemma~4.3]{Duke}, the Bousfield localization $\Loc_S\bbD$
of a \emph{triangulated} derivator $\bbD$ remains triangulated as
long as $S$ is stable under the loop space functor. For more
general~$S$, to remain in the world of \emph{triangulated}
derivators, one has to localize with respect to the set $\Omega(S)$
generated by $S$ and loops, as follows.
\begin{proposition}\label{prop:omega}
Let $\bbD$ be a triangulated derivator and $S$ a class of morphisms
in~$\bbD(e)$. Let us denote by $\Omega(S)$ the smallest class of
morphisms in $\bbD(e)$ which contains $S$ and is stable under the
loop space functor $\Omega:\bbD(e) \rightarrow \bbD(e)$. Then for
any \emph{triangulated} derivator~$\bbT$, we have an equality of
categories
\begin{equation*}
\uHom_{!, \Omega(S)}(\bbD,\bbT) = \uHom_{!,S}(\bbD,\bbT)\,.
\end{equation*}
As a consequence, whenever $\Loc_{\Omega(S)}\bbD$ exists,
this is the \emph{triangulated} left Bousfied localization of $\bbD$ with respect to~$S$.
\end{proposition}

\begin{proof}
For $F$ an element of $\uHom_{!}(\bbD,\bbT)$, the functor
$F(e):\bbD(e)\to\mathbb{T}(e)$ commutes with homotopy colimits,
hence it commutes in particular with the suspension functor. Since
both $\bbD$ and $\bbT$ are triangulated, suspension and loop space
functors are inverse to each other. Hence $F(e)$ also commutes
with~$\Omega$. It is then obvious that $F(e)$ sends $S$ to
isomorphisms if and only if it does so with~$\Omega(S)$.
\end{proof}

\begin{theorem}[Dugger]\label{dugger1}
Let $\cM$ be a combinatorial model category.
Then, there exists a small category $A$ and a small set of maps $S$
in $\mathsf{Hot}_A(e)=\Ho(\spref A)$,
such that $\HO(\cM)$ is equivalent to $\Loc_S\mathsf{Hot}_A$.
\end{theorem}

\begin{proof}
This follows from \cite[Proposition 3.3]{Duggerbis}
and from Theorem \ref{thm:Quillenloc} applied to the
projective model structure on the category of simplicial
presheaves of a small category.
\end{proof}

\begin{remark}
It follows immediately from Theorem~\ref{dugger1} that the statement of Theorem \ref{thm:Quillenloc}
holds also for left proper combinatoriel model categories.
In particular, any derivator which is equivalent to a derivator
associated to a combinatorial model category admits a left Bousfield localization with respect to any small set of maps.
\end{remark}

\begin{proposition}\label{derivatorleftlocmonoidal}
Let $\bbD$ be a symmetric monoidal derivator, and $S$
a class of maps in $\bbD(e)$. Assume that $S$ is closed under
tensor product in $\bbD$, and that the left Bousfield localization
of $\bbD$ by $S$ exists. Then, $\Loc_S\bbD$ is
symmetric monoidal, and the localization morphism
$\gamma:\bbD\too\Loc_S\bbD$ is symmetric monoidal.
Moreover, given any symmetric monoidal
derivator $\bbD'$, the induced functor
$$\gamma^*:\uHom^\otimes_!(\Loc_S\bbD,\bbD')\too
\uHom^\otimes_!(\bbD,\bbD')$$
is fully-faithful, and its essential image
consists of the symmetric monoidal
homotopy colimit preserving morphisms which
send $S$ to isomorphisms.
\end{proposition}
\begin{proof}
This is an immediate consequence of the universal property
of $\Loc_S\bbD$. The details are left as an exercise for the reader.
\end{proof}
\subsection{Stabilization}\label{sub:Heller}
Let $\bbD$ be a derivator. Then there is a \emph{universal
pointed derivator} $\bbD\to\bbD_\bullet$\,: given any pointed derivator $\bbD'$, the induced functor
$$\uHom_!(\bbD_\bullet,\bbD')\stackrel{\sim}{\too}\uHom_!(\bbD,\bbD')$$
is an equivalence of categories; see \cite[Corollaire~4.19]{propuni}.
Using the explicit construction of $\bbD_\bullet$ (see \cite[\S4.5]{propuni})
it is easy to see that when $\bbD_\bullet$ is strong (resp. regular)
so is $\bbD$. If $\bbD=\HO(\cM)$ for some model category $\cM$, then
$\bbD_\bullet$ is equivalent to $\HO(\cM_\bullet)$, where $\cM_\bullet$
denotes the model category of pointed objects in $\cM$; see \cite[Proposition~1.1.8]{Hovey}.

Let $\bbD$ be a pointed derivator. A \emph{stabilization} of $\bbD$
is a homotopy colimit preserving morphism $\stab:\bbD\to\St(\bbD)$, with $\St(\bbD)$ a triangulated strong derivator, which is universal
for these properties\,: given any
triangulated strong derivator $\bbT$, the induced functor
$$ \stab^{\ast}:\HomC(\St(\bbD),\bbT) \stackrel{\sim}{\too} \HomC(\bbD,\bbT)\,.$$
is an equivalence of categories.

\begin{theorem}[Heller~\cite{Heller}]\label{thm:Heller}
Any pointed regular strong derivator admits a stabilization.
\end{theorem}

Given a pointed simplicial model category $\cM$,
the second named author compared in \cite[\S8]{Duke} the
derivator associated to the model category $\Spt^{\bbN}(\cM)$ of
${S^1}$-spectra on $\cM$ with the stabilization of
the derivator associated to the model category $\cM$.

\begin{proposition}\label{prop:comparison}
Let $\cM$ be a pointed, simplicial, left proper, cellular model category.
Assume that sequential homotopy colimits commute with finite products
and homotopy pullbacks. Then, the induced morphism of triangulated strong derivators
$$ \St(\HO(\cM)) \stackrel{\sim}{\too} \HO(\Spt^{\bbN}(\cM))$$
is an equivalence.
\end{proposition}

\begin{proof}
See \cite[Theorem~8.7]{Duke}.
\end{proof}

Let $\bbD$ be a pointed strong derivator (see \S\ref{sub:properties})
and $S$ be a class of morphisms in $\bbD(e)$. Assume that $\bbD$ admits a left Bousfield localization
$\Loc_{S}\bbD$ with respect to~$S$; see \S\ref{sub:leftBousfield}.
Assume also that the stablization $\St(\bbD)$ of $\bbD$ exists.

We then have two homotopy colimit preserving morphisms\,:
$$\Loc_S\bbD\xleftarrow{\ \gamma \ }\bbD\xrightarrow{\stab}\St(\bbD)\, .$$
By examining the relevant universal properties, we obtain the following result.

\begin{proposition}\label{commut}
Under the above assumptions, the derivator
$\Loc_{\Omega(\stab(S))}\St(\bbD)$ exists if and only if the derivator
$\St(\Loc_S\bbD)$ exists. Moreover, if this is the case, then
$\Loc_{\Omega(\stab(S))}\St(\bbD)\simeq \St(\Loc_S\bbD)$
under $\bbD$.
\end{proposition}

\begin{corollary}\label{stabcombmodcat}
Let $\cM$ be a pointed left proper combinatorial model category.
Then the stabilization $\St(\HO(\cM))$ of $\HO(\cM)$
exists and is equivalent to the
derivator associated to a stable combinatorial model
category.
\end{corollary}

\begin{proof}
Thanks to Theorem \ref{dugger1}
we may assume that $\HO(\cM)\simeq \Loc_S\mathsf{Hot}_{A,\bullet}$.
Using Proposition \ref{commut} and Theorem \ref{thm:Quillenloc},
we see it is sufficient to treat the case where $S$ is the empty set. Proposition~\ref{prop:comparison} allow us then to conclude the proof.
\end{proof}

\begin{remark}
A careful analysis of the proof of the Corollary~\ref{stabcombmodcat}
will lead to a proof of Proposition \ref{prop:comparison}
for any simplicial combinatorial model category $\cM$. Note that since any combinatorial model category is
equivalent to a simplicial one, we conclude the existence of stabilizations for any derivator associated to
a combinatorial model category; however, we will not need
this level of generality.
\end{remark}

\begin{theorem}\label{stableDay}
Let $A$ be a small symmetric monoidal category.
Then there is a unique symmetric monoidal structure
on the triangulated derivator $\St(\mathsf{Hot}_{A,\bullet})$
whose tensor product preserves homotopy colimits in each
variables, such that the composed morphism
$$\underline{A}\too \mathsf{Hot}_{A}\too\St(\mathsf{Hot}_{A,\bullet})$$
is symmetric monoidal. Moreover, given any strong
triangulated derivator $\bbD$, the induced functor
$$\uHom^\otimes_!(\St(\mathsf{Hot}_{A,\bullet}),\bbD)
\stackrel{\sim}{\too} \uHom^\otimes (\underline{A},\bbD)$$
is an equivalence of categories.
\end{theorem}

\begin{proof}
It follows immediately from
Theorem \ref{derKanext1} and from the universal property
of $\St(\mathsf{Hot}_{A,\bullet})$ that,
for any small category $A$ and any strong triangulated
derivator $\bbD$, we have canonical equivalences of categories:
$$\uHom_!(\St(\mathsf{Hot}_{A,\bullet}),\bbD)\simeq
\uHom(\underline{A},\bbD)\simeq \bbD(A^\op) \, .$$
Starting from this point on, the proof of Theorem \ref{monoidalderKanext}
holds here \emph{mutatis mutandis}.
\end{proof}

\begin{remark}
The derivator $\St(\mathsf{Hot}_{A,\bullet})$ is equivalent
to the derivator associated to the model category of
symmetric spectra in the category of pointed simplicial
presheaves on $A$. This equivalence defines a symmetric monoidal
structure on $\St(\mathsf{Hot}_{A,\bullet})$, making the
morphism
$$\underline{A}\too \mathsf{Hot}_{A}\too\St(\mathsf{Hot}_{A,\bullet})$$
symmetric monoidal; see Proposition~\ref{prop:Mprojectivepointed}, Theorem~\ref{thm:main-Hovey},
and Proposition~\ref{prop:comparison}. This monoidal structure coincides with the one of
Theorem~\ref{stableDay}, thanks to the uniqueness of the latter.
\end{remark}

\subsection{Spectral enrichment}\label{sub:spectral-enr}
Recall from \cite[Appendix~A.3]{CT} that any triangulated derivator
$\bbD$ is canonically enriched over spectra, \ie we have a morphism of derivators
$$ \bbR\Hom(-,-): \bbD^\op \times \bbD \too \HO(\Spt^{\bbN})\,.$$
Moreover, this enrichment over spectra is compatible with adjunctions\,: given an adjunction
$$
\xymatrix{
\bbD' \ar@<1ex>[d]^{\Psi} \\
\bbD \ar@<1ex>[u]^{\Phi}
}
$$
we have a canonical isomorphism in the stable homotopy category of spectra\,:
\begin{eqnarray*}
\bbR\Hom_{\bbD'}(\Phi  X,Y) \simeq \bbR \Hom_{\bbD}(X, \Psi Y) && X \in \bbD,\,\, Y \in \bbD'\,.
\end{eqnarray*}
\backmatter

\end{document}